\documentclass{lmcs}
\pdfoutput=1

\usepackage{lastpage}
\lmcsdoi{15}{3}{13}
\lmcsheading{}{\pageref{LastPage}}{}{}%
{Sep.~20,~2017}{Aug.~08,~2019}{}

\usepackage{amstext}
\usepackage{amssymb}
\usepackage{amsfonts}
\usepackage{amsmath}
\usepackage{amsthm}

\usepackage{stmaryrd}

\usepackage{booktabs}

\def\squareforqed{\hbox{\rlap{$\sqcap$}$\sqcup$}}
\def\qed{\ifmmode\squareforqed\else{\unskip\nobreak\hfil
\penalty50\hskip1em\null\nobreak\hfil\squareforqed%
\parfillskip=0pt\finalhyphendemerits=0\endgraf}\fi}

\def\cal#1{{\mathcal{#1}}}

\def\bsigma#1{\mathbf{\Sigma}^0_{#1}}
\def\bpi#1{\mathbf{\Pi}^0_{#1}}

\def\Bold2{{\bf 2}}

\def\calA{{\mathcal{A}}}
\def\calB{{\mathcal{B}}}

\def\calD{{\mathcal{D}}}
\def\calF{{\mathcal{F}}}
\def\calH{{\mathcal{H}}}
\def\calI{{\mathcal{I}}}
\def\calK{{\mathcal{K}}}

\def\calP{{\mathcal{P}}}
\def\calS{{\mathcal{S}}}
\def\calU{{\mathcal{U}}}

\def\PL{{\mathrm{\mathbf{A}}}}
\def\PU{{\mathrm{\mathbf{K}}}}
\def\PO{{\mathrm{\mathbf{O}}}}
\def\PV{{\mathrm{\mathbf{L}}}}

\def\frakA{{\mathfrak{A}}}
\def\frakK{{\mathfrak{K}}}
\def\frakU{{\mathfrak{U}}}

\newcommand{\uparw}{{\uparrow}}
\newcommand{\dnarw}{{\downarrow}}

\newcommand{\triangleup}{{\vartriangle}}

\def\akka{{\sigma}}
\def\kaak{{\tau}}
\def\kaoo{{\phi}}
\def\ooka{{\psi}}
\def\aook{{\alpha}}
\def\okao{{\beta}}
\def\kooa{{\gamma}}
\def\oako{{\delta}}

\begin{document}

\title{On the commutativity of the powerspace constructions}

\author[M. de Brecht]{Matthew de Brecht\rsuper{a}}
\address{\lsuper{a}Graduate School of Human and Environmental Studies, Kyoto University, Japan}
\email{matthew@i.h.kyoto-u.ac.jp}

\author[T. Kawai]{Tatsuji Kawai\rsuper{b}}
\address{\lsuper{b}Japan Advanced Institute of Science and Technology, Japan}
\email{tatsuji.kawai@jaist.ac.jp}

\begin{abstract}
We investigate powerspace constructions on topological spaces, with a particular focus on the category of quasi-Polish spaces. We show that the upper and lower powerspaces commute on all quasi-Polish spaces, and show more generally that this commutativity is equivalent to the topological property of consonance. We then investigate powerspace constructions on the open set lattices of quasi-Polish spaces, and provide a complete characterization of how the upper and lower powerspaces distribute over the open set lattice construction.
\end{abstract}


\dedicatory{This paper is dedicated to the memory of Klaus Keimel.}

\maketitle

\section{Introduction}

Given a topological space $X$, the \emph{lower powerspace} $\PL(X)$ is the set of closed subsets of $X$ with the lower Vietoris topology, the \emph{upper powerspace} $\PU(X)$ is the set of (saturated) compact subsets of $X$ with the upper Vietoris topology, and the \emph{convex} (or \emph{Vietoris}) \emph{powerspace} $\PV(X)$ is the set of lenses in $X$ with the Vietoris topology.

If $X$ is Hausdorff, then the notion of lens and compact subset become equivalent, and $\PV(X)$ is the well-known hyperspace of compact subsets often used by topologists (see Section 4.F of~\cite{ke95}). This is also the same Vietoris construction used in~\cite{KKV04}  (but restricted to Stone spaces) to provide coalgebraic semantics for modal logics. $\PL(X)$ and $\PU(X)$ are generally non-Hausdorff and might seem less familiar to some topologists, but $\PL(X)$ should be easily recognized as the topological space underlying the well-known Effros Borel space (see Section 12.C in~\cite{ke95}), at least when $X$ is countably based.

In domain theory, the upper and lower powerdomains (see Section IV-8 of~\cite{g03}) coincide with the upper and lower powerspaces for continuous dcpos with the Scott topology, although the constructions differ in general for non-continuous dcpos. Under some additional mild assumptions, the Plotkin powerdomain and convex powerspace are also known to coincide. Powerdomains are often used in theoretical computer science to model the semantics of non-deterministic programs.

In this paper, we investigate the powerspace constructions with a focus on the category of quasi-Polish spaces~\cite{dbr}, which is general enough to contain all Polish spaces and all $\omega$-continuous domains. We show that the powerspace constructions preserve the property of being quasi-Polish, and as a result we can view the constructions as monads on the category of quasi-Polish spaces. We will also provide topological characterizations of when the powerspace constructions commute, and investigate the interaction between the powerspace constructions and the contravariant endofunctor $\PO$ which maps a space $X$ to its lattice of open subsets $\PO(X)$ equipped with the Scott-topology.

The locale theoretic versions of the powerspace constructions are known as \emph{powerlocales} and have a very rich theory~\cite{J85,V89,JV91,schalk:phd,V04,VT04}. Although the powerspace and powerlocale constructions behave differently in general, some of the more notable differences (such as the failure of the commutativity of upper and lower powerspaces) disappear when one restricts to powerspace constructions on quasi-Polish spaces. As we will further explain below, several of our topological results for the powerspaces on quasi-Polish spaces correspond to well-known results for powerlocales, and in some cases our results can be easily obtained using known results and techniques from locale theory. However, we expect that the topological proofs provided here will be more accessible to those who are unfamiliar with locale theory, even though they come at the expense of using classical logic.

In Sections~\ref{sec:lowerpowerspace} through~\ref{sec:convexpowerspace} we show that the lower, upper, and convex powerspace constructions preserve the property of being quasi-Polish. Although we use straightforward topological proofs, the preservation results for the lower and upper powerspaces can also be obtained using known results in locale theory. R.~Heckmann showed in~\cite{H15} that every countably presented locale is spatial and that the corresponding spaces are precisely the quasi-Polish spaces. Furthermore, results by S.~Vickers showing the geometricity of the lower and upper powerlocales (see Theorems 7.5 and 8.8 in~\cite{V04}) demonstrate that the powerlocale constructions preserve the property of having a countable presentation. Finally, it is known (see Propositions 6.26 and 7.39 in~\cite{schalk:phd}) that if $X$ is a sober space, then the lower (or upper) powerspace of $X$ is homeomorphic to the space of points of the corresponding powerlocale. Combining these results provides a locale theoretic proof that the lower and upper powerspace of a quasi-Polish space is quasi-Polish.

In Section~\ref{sec:consonance}, we will show that a topological property known as \emph{consonance} (see~\cite{dgl95,NS96,cw98,bo99}) is equivalent to the commutativity of the upper and lower powerspaces in the sense that $\PL(\PU(X)) \cong \PU(\PL(X))$ under a naturally defined homeomorphism. Analogous commutativity results are known in both domain theory and locale theory. In domain theory, K.E.~Flannery and J.J.~Martin showed in~\cite{FM90} that the upper and lower powerdomain constructions commute for all bounded complete algebraic dcpos. R.~Heckmann~\cite{H91} later extended this result to all dcpos. In~\cite{H92}, R.~Heckmann proved a commutativity result for the lower powerdomain and a modified version of the upper powerdomain using topological methods that are more closely related to the approach we take here. In locale theory, P.T.~Johnstone and S.~Vickers showed in~\cite{JV91} that the upper and lower powerlocales commute for all locales. It can be shown that the upper powerlocale of a non-consonant space is not spatial, which explains why the powerlocales can commute even in the cases where the powerspaces do not.

Every quasi-Polish space is consonant, and the functions in Definition~\ref{def:AK_KA} can be interpreted as distributive laws (in the sense of Beck~\cite{be69}) between the upper and lower powerspace monads on the category of quasi-Polish spaces.  The locale theoretic version of this observation has already been made by S.~Vickers in~\cite{V04}. At the time of this writing we have been unable to verify whether the upper and lower powerspace constructions preserve consonance, which prevents us from extending this result to the more general category of consonant spaces.

Section~\ref{sec:consonance} also provides connections between the composition of the upper and lower powerspace constructions with the double powerspace construction. Given a topological space $X$, we define the double powerspace of $X$ to be $\PO(\PO(X))$. In general, the contravariant endofunctor $\PO$ does not preserve the property of being quasi-Polish. However, the double powerspace construction $\PO\circ\PO$ does restrict to a covariant endofunctor on the category of quasi-Polish spaces, and is naturally isomorphic to the composition of the upper power space monad with the lower powerspace monad (or vice versa) in this case. These results are analogous to the properties of the lower, upper and double powerlocales investigated in~\cite{V04,VT04}, and show that consonance is a key property needed for the spatial and localic theories of the power constructions to agree.

If $X$ is quasi-Polish then $\PO(X)$ is quasi-Polish if and only if $X$ is locally compact. If $X$ is quasi-Polish but not locally compact, such as the Baire space, then $\PO(X)$ will not even be countably based. From R.~Heckmann's results we know that spaces of the form $\PO(X)$ for quasi-Polish $X$ are precisely the countably presented frames equipped with the Scott-topology. We investigate the powerspace constructions on these spaces in Sections~\ref{sec:AO_OK} and~\ref{sec:KO_OA} and show that $\PL(\PO(X))\cong \PO(\PU(X))$ and $\PU(\PO(X))\cong \PO(\PL(X))$ under natural homeomorphisms whenever $X$ is quasi-Polish. Together, these results provide a complete picture of how $\PL$, $\PU$, and $\PO$ interact on quasi-Polish spaces (see a summary in Table~\ref{tab:summary} below). E.~Neumann~\cite{N18} has recently shown that computable generalizations of these results hold within the framework of computable analysis.\footnote{Note that the powerspaces used in computable analysis have the \emph{sequentialization} of the upper and lower Vietoris topologies used in this paper, which is a subtle but relevant detail when dealing with non-countably based spaces such as $\PO(X)$ for non-locally compact $X$. However, the homeomorphisms we prove in Theorems~\ref{thrm:AO_OK_homeomorphic} and~\ref{thrm:KO_OA_homeomorphic} imply (using Proposition~2.2(6) of~\cite{Sch:HMT}) that the upper and lower Vietoris topologies on the powerspaces of $\PO(X)$ are already sequential when $X$ is quasi-Polish. }



\begin{table}%
\label{tab:summary}
\begin{tabular}{rl}
\multicolumn{2}{c}{$\PL(\PU(X)) \cong \PU(\PL(X))$}\\
\multicolumn{2}{c}{(see Definition~\ref{def:AK_KA} and Theorem~\ref{thrm:cons_AK_KA_homeo})}\\
&\\
\quad$\akka_X\colon \PL(\PU(X)) \to \PU(\PL(X))$; & \quad$\akka_X(\calA)= \{ A\in \PL(X) \,|\, (\forall K\in \calA)\, A\cap K\not=\emptyset\}$,\\[.25em]
&\quad$\akka_X^{-1}(\Box\Diamond U) = \Diamond\Box U$ (for $U\in\PO(X)$)\\[.5em]
$\kaak_X\colon \PU(\PL(X)) \to \PL(\PU(X))$; & \quad$\kaak_X(\calK)= \{ K\in \PU(X) \,|\, (\forall A\in \calK)\, A\cap K\not=\emptyset\}$,\\[.25em]
&\quad$\kaak_X^{-1}(\Diamond\Box U) = \Box\Diamond U$ (for $U\in\PO(X)$)\\
&\\
\midrule
&\\
\multicolumn{2}{c}{$\PU(\PL(X)) \cong \PO(\PO(X))$}\\
\multicolumn{2}{c}{(see Definition~\ref{def:KAOO} and Theorem~\ref{thrm:scott_top_basis_for_OOX})}\\
&\\
$\kaoo_X\colon \PU(\PL(X)) \to \PO(\PO(X))$; &\quad $\kaoo_X(\calK)= \{U\in\PO(X) \,|\, \calK\in\Box\Diamond U\}$,\\[.25em]
&\quad$\kaoo_X^{-1}(\boxtimes U) = \Box\Diamond U$ (for $U\in\PO(X)$)\\[.5em]
$\ooka_X\colon \PO(\PO(X)) \to \PU(\PL(X))$; & \quad$\ooka_X(\calH)= {\bigcap}_{U\in\calH}\Diamond U$,\\[.25em]
&\quad$\ooka_X^{-1}(\Box\Diamond U) = \boxtimes U$ (for $U\in\PO(X)$)\\[.5em]
&\\
\midrule
&\\
\multicolumn{2}{c}{$\PL(\PO(X)) \cong \PO(\PU(X))$}\\
\multicolumn{2}{c}{(see Definition~\ref{def:AO_OK} and Theorem~\ref{thrm:AO_OK_homeomorphic})}\\
&\\
$\aook_X\colon \PL(\PO(X))\to\PO(\PU(X))$; & \quad $\aook_X(\frakA)= {\bigcup}_{U\in\frakA}\square U$,\\[.25em]
&\quad$\aook_X^{-1}(\triangleup \calA) = \Diamond\kaoo_X(\akka_X(\calA))$ (for $\calA \in \PL(\PU(X))$)\\[.5em]
$\okao_X\colon \PO(\PU(X))\to\PL(\PO(X))$; & \quad $\okao_X(\calU)= \{U\in\PO(X) \mid \square U \subseteq \calU \}$,\\[.25em]
&\quad$\okao^{-1}_X(\Diamond \calH) = \triangleup \kaak_X(\ooka_X(\calH))$ (for $\calH\in\PO(\PO(X))$)\\[.5em]
&\\
\midrule
&\\
\multicolumn{2}{c}{$\PU(\PO(X)) \cong \PO(\PL(X))$}\\
\multicolumn{2}{c}{(see Definition~\ref{def:KO_OA} and Theorem~\ref{thrm:KO_OA_homeomorphic})}\\
&\\
$\kooa_X\colon \PU(\PO(X))\to\PO(\PL(X))$; &\quad $\kooa_X(\frakK)= {\bigcap}_{U\in\frakK} \Diamond U$,\\[.25em]
&\quad$\kooa^{-1}_X(\triangledown\calK) = \Box\kaoo_X(\calK)$ (for $\calK\in\PU(\PL(X))$)\\[.5em]
$\oako_X\colon \PO(\PL(X)) \to \PU(\PO(X))$; &\quad $\oako_X(\calU)= \{U \in\PO(X) \mid \calU \subseteq \Diamond U\}$,\\[.25em]
&\quad$\oako^{-1}_X(\Box\calH) = \triangledown\ooka_X(\calH)$ (for $\calH\in\PO(\PO(X))$)\\[.5em]
&\\
\end{tabular}
\caption{A summary of the homeomorphisms investigated in this paper. All of these results are valid when $X$ is quasi-Polish. The results concerning $\PL(\PU(X)) \cong \PU(\PL(X))$ are valid when $X$ is consonant. The results for $\PU(\PL(X)) \cong \PO(\PO(X))$ are valid when $X$ is countably based.}
\end{table}


\section{Preliminaries}

A topological space is \emph{quasi-Polish} if and only if it is countably based and the topology is induced by a Smyth-complete quasi-metric (see~\cite{dbr}). Several equivalent characterizations of quasi-Polish spaces exist, but for this paper the most relevant characterization is that a space is quasi-Polish if and only if it is homeomorphic to a $\bpi 2$-subset of an $\omega$-algebraic domain (i.e., a countably based algebraic domain). A subset $A$ of a topological space $X$ is a \emph{$\bpi 2$-subset} if there exist sequences ${(U_i)}_{i\in\omega}$ and ${(V_i)}_{i\in\omega}$ of open subsets of $X$ satisfying
\[x\in A  \text{ if and only if } (\forall i\in\omega)[x\in U_i \Rightarrow x\in V_i].\]
We denote the collection of $\bpi 2$-subsets of $X$ by $\bpi 2(X)$. If $X$ is quasi-Polish then a subspace $A$ of $X$ is quasi-Polish if and only if $A \in\bpi 2(X)$.

The \emph{specialization preorder} on a topological space $X$ is defined as $x\leq y$ if and only if $x$ is in the closure of the singleton set $\{y\}$. This is a partial order if and only if $X$ is a $T_0$-space, in which case it is called the \emph{specialization order}. Given a subset $A\subseteq X$, we define $\uparw A = \{ y\in X \mid (\exists x\in A)\, x \leq y\}$ and $\dnarw A = \{ y\in X \mid (\exists x\in A)\, y \leq x\}$. Note that the argument to $\uparw$ and $\dnarw$ is always a set, and when $A=\{x\}$ is a singleton we do not abbreviate $\uparw A$ by $\uparw x$ in order to avoid potential ambiguities when working with powerspaces.

We refer the reader to~\cite{g03} and~\cite{GL13} for background on domain theory and the Scott topology on partially ordered sets. A.~Schalk's thesis~\cite{schalk:phd} is a valuable source of information on the powerspace constructions for general topological spaces, as well as the corresponding powerdomain and powerlocale constructions.

\section{Lower powerspaces}\label{sec:lowerpowerspace}

The \emph{lower powerspace} $\PL(X)$ is defined to be the set of closed subsets of $X$ with the lower Vietoris topology. The lower Vietoris topology is generated by sets of the form $\Diamond U = \{ A\in\PL(X) \,|\, A\cap U \not=\emptyset\}$ for open $U\subseteq X$. Note that the specialization order on $\PL(X)$ is subset inclusion.

It is well known that $\PL$ is a monad on the category of topological spaces (see Section~6.3 in~\cite{schalk:phd}):
\begin{itemize}
\item
Each continuous function $f\colon Y\to X$ maps to $\PL(f) \colon \PL(Y)\to \PL(X)$ defined as $\PL(f)(A) = Cl_X(f(A))$ (the closure of the image of $A$ under $f$),
\item
The unit $\eta^\PL_X\colon X\to \PL(X)$ maps $x$ to $\dnarw \{x\}$,
\item
The multiplication $\mu^\PL_X \colon \PL(\PL(X))\to \PL(X)$ maps $\calA$ to $\bigcup_{A\in\calA} A$.
\end{itemize}

\noindent
Note that ${(\eta^\PL_X)}^{-1}(\Diamond U) = U$ and ${(\mu^\PL_X)}^{-1}(\Diamond U) = \Diamond\Diamond U$ and ${(\PL(f))}^{-1}(\Diamond U) = \Diamond f^{-1}(U)$ for every open $U\subseteq X$ and continuous $f\colon Y \to X$.

Given a partially ordered set $P$, the \emph{weak topology} on $P$ is generated by open sets of the form $\{ y\in P \mid y\not\leq x\}$, where $x$ varies over elements of $P$. See Proposition 1.7 and Section 6.3 in~\cite{schalk:phd} for the following.

\begin{prop}\label{prop:lowerpowerspace_sober}
$\PL(X)$ is sober and the lower Vietoris topology coincides with the weak topology.
\qed%
\end{prop}

\begin{exa}\label{ex:AX_not_Scott}
The topology on $\PL(X)$ is not the Scott topology in general, even when $X$ is quasi-Polish. As a counterexample, consider the set $X = \{ \infty \} \cup \omega$ partially ordered so that every element is less than or equal to $\infty$, but all other elements are incomparable. $X$ becomes a quasi-Polish space when given the weak topology, which can be seen by noting that $X$ is homeomorphic to the $\bpi 2$-subspace $\{ S\subseteq \omega \mid (\forall n,m\in\omega)\,n\not=m \Rightarrow (n\in S\vee m\in S)\}$ of $\calP(\omega)$ (where $\calP(\omega)$ is the powerset of the natural numbers with the Scott-topology). $\PL(X)$ consists of $\emptyset$, $X$, and all finite subsets of $\omega$. The subset of $\PL(X)$ of all sets containing at least two elements is Scott-open in $\PL(X)$, but it is not open with respect to the lower Vietoris topology because for each  non-empty open $U\subseteq X$ the subbasic open set $\Diamond U \subseteq \PL(X)$ contains cofinitely many closed singletons.
\qed%
\end{exa}

Although we will not need it later, we also mention the following:

\begin{prop}
If $X$ is a countably based $T_0$-space then $\eta^\PL_X(X) \in \bpi 2(\PL(X))$ if and only if $X$ is sober.
\end{prop}
\begin{proof}
First assume $X$ is sober and has a countable basis $\calB$. Sobriety implies that $\eta^\PL_X(X)$ is precisely the subset of irreducible closed subsets of $X$. Furthermore, $A\in\PL(X)$ is irreducible if and only if $A$ is non-empty (i.e. $A\in \Diamond X$) and
\[(\forall U,V\in\cal B) A \in (\Diamond U \cap \Diamond V) \implies A \in \Diamond(U\cap V),\]
which clearly defines a $\bpi 2$-subset of $\PL(X)$.

For the converse, simply note that $\PL(X)$ is always sober, that $X$ is homeomorphic to $\eta^\PL_X(X)$, and that every $\bpi 2$-subset of a sober space is sober~\cite{d1?}.
\end{proof}

A space is \emph{Baire} if and only if countable intersections of dense open sets are dense. Equivalently, a space is Baire if and only if countable intersections of dense $\bpi 2$-sets are dense. The equivalence of these two definitions is not entirely obvious for non-metrizable spaces, but it is essentially the content of Theorem~2.24 of~\cite{H15}. This version of the Baire category theorem was also independently discovered in~\cite{BG15} (Theorem~3.14). The dual statement (in terms of $\bsigma 2$-sets) is proven in~\cite{d1?} (Lemma~5.1) for (non-empty) Baire spaces, and an alternative proof that applies to (non-empty) quasi-Polish spaces is given in~\cite{dbr14} (Lemma~30).

A space is \emph{completely Baire} if every closed subspace is a Baire space. Completely Baire spaces play a central role in the spatiality results of~\cite{H15} (where the term \emph{(extended) local Baire property} is used), and these spaces are further investigated in~\cite{d1?}. It was shown in~\cite{dbr} that every quasi-Polish space is a Baire space, and therefore every quasi-Polish space is completely Baire because every closed subspace of a quasi-Polish space is quasi-Polish.

\begin{prop}\label{prop:lower_preserves_pi2}
Let $Y$ be a countably based completely Baire space. If $e\colon X\to Y$ is a topological embedding of $X$ into $Y$ as a $\bpi 2$-subspace, then $\PL(e)\colon \PL(X)\to \PL(Y)$ is a topological embedding of $\PL(X)$ into $\PL(Y)$ as a $\bpi 2$-subspace.
\end{prop}
\begin{proof}
Assume $Y$ is a countably based completely Baire space with countable basis $\calB$. Assume $X\in \bpi 2(Y)$, and let $U_i,V_i$ $(i\in\omega)$ be open subsets of $Y$ such that $x\in X$ if and only if $(\forall i\in\omega)[x\in U_i \Rightarrow x\in V_i]$.

Let $e\colon X\to Y$ be the embedding of $X$ into $Y$, and let $f = \PL(e) \colon \PL(X) \to \PL(Y)$. By definition, $f$ maps $A\in \PL(X)$ to $Cl_Y(A)\in\PL(Y)$. Since $e$ is an embedding, for any open $W\subseteq X$ there is open $W'\subseteq Y$ such that $W = e^{-1}(W')$, hence $\Diamond W = \Diamond e^{-1}(W') = {(\PL(e))}^{-1}(\Diamond W') = f^{-1}(\Diamond W')$. It follows that $f$ is a topological embedding of $\PL(X)$ into $\PL(Y)$.

Let $\calS$ be the subset of all $A\in \PL(Y)$ satisfying $A\in\Diamond(B \cap U_i) \Rightarrow A\in \Diamond(B \cap V_i)$ for all $i\in\omega$ and $B\in\cal B$. Clearly $\calS\in\bpi 2(\PL(Y))$, hence the theorem will be proved by showing that $\calS = range(f)$.

First we show $range(f)\subseteq \calS$. Fix $A\in\PL(X)$. Assume $i\in\omega$ and $B\in \calB$ are such that $f(A)\in\Diamond(B\cap U_i)$. Since $f(A) = Cl_Y(A)$ there exists $x\in A \cap B \cap U_i$, so the definition of $X$ and the assumption $A\subseteq X$ imply $x \in  A\cap B \cap V_i$. Thus $f(A)\in \Diamond(B\cap V_i) $, and it follows that $f(A)\in \calS$.

Finally, we show that $\calS \subseteq range(f)$. Fix $A\in \calS$. We prove that $X\cap A$ is dense in $A$, hence $A = f(A\cap X)= Cl_Y(A\cap X)$. Since $X = \bigcap_{i\in\omega} ((Y\setminus U_i) \cup V_i)$ and $Y$ is completely Baire, it suffices to show that $(Y\setminus U_i) \cup V_i$ is dense in $A$ for each $i\in\omega$. Let $B\in\calB$ be a basic open set such that $A\cap B \not=\emptyset$. If $A\cap B \cap (Y\setminus U_i) = \emptyset$, then $A\cap B \subseteq U_i$ hence $A\in\Diamond(B\cap U_i)$. From the assumption $A\in \calS$ we obtain $A\in \Diamond(B\cap V_i)$, and it follows that $(Y\setminus U_i) \cup V_i$ is dense in $A$.
\end{proof}


As shown by D.~Scott~\cite{S76} (see also~\cite{dbr}), a topological embedding $e\colon X \to Y$ which identifies $X$ with a $\bpi 2$-subspace of $Y$ is actually an equalizer for a pair of continuous functions $f,g\colon Y\to \calP(\omega)$, where $\calP(\omega)$ is the powerset of the natural numbers with the Scott topology. Therefore, the above proposition would be trivial if the functor $\PL$ preserved equalizers. However, it is easy to see that $\PL$ does not preserve equalizers in general. As a simple counterexample, let $\Bold2=\{0,1\}$ have the discrete topology, define $f$ to be the identity on $\Bold2$, and let $g$ be defined as $g(0)= 1$ and $g(1)=0$. The equalizer of $f$ and $g$ is the embedding $e\colon \emptyset \to \Bold2$ of the empty subspace $\emptyset$ into $\Bold2$. Note that $\PL(\emptyset) = \{\emptyset\}$ has exactly one point, whereas the equalizer for $\PL(f)$ and $\PL(g)$ is the two-point subspace $\{\emptyset, \Bold2\}$ of $\PL(\Bold2)$. Therefore, $\PL(e)$ is not the equalizer of $\PL(f)$ and $\PL(g)$.


\begin{thm}\label{thrm:lowerpowerspace}
If $X$ is quasi-Polish then $\PL(X)$ is quasi-Polish.
\end{thm}
\begin{proof}
We can assume without loss of generality that $X$ is a $\bpi 2$-subspace of some $\omega$-algebraic domain $D$ (in particular, we can always take $D=\calP(\omega)$). It is well known (see, e.g., the results by M.~B.~Smyth~\cite{sm83} and A.~Schalk~\cite{schalk:phd}) that $\PL(D)$ is an $\omega$-algebraic domain. Every $\omega$-algebraic domain is quasi-Polish hence completely Baire, so Proposition~\ref{prop:lower_preserves_pi2} implies $\PL(X)$ is homeomorphic to a $\bpi 2$-subspace of $\PL(D)$. Therefore, $\PL(X)$ is quasi-Polish.
\end{proof}

\section{Upper powerspaces}\label{sec:upperpowerspace}

The \emph{upper powerspace} $\PU(X)$ is defined to be the set of compact saturated subsets of $X$ with the upper Vietoris topology. A subset of a topological space is \emph{saturated} if and only if it is equal to the intersection of all of its open neighborhoods. The upper Vietoris topology is generated by sets of the form $\Box U = \{ K\in\PU(X) \,|\, K\subseteq U\}$ for open $U\subseteq X$. Note that the specialization order on $\PU(X)$ is \emph{reverse} subset inclusion.

It is well known that $\PU$ is a monad on the category of topological spaces (see Section~7.3 in~\cite{schalk:phd}):
\begin{itemize}
\item
Each continuous function $f\colon Y\to X$ maps to $\PU(f) \colon \PU(Y)\to \PU(X)$ defined as $\PU(f)(K) = \uparw f(K)$ (the saturation of the image of $K$ under $f$),
\item
The unit $\eta^\PU_X \colon X\to \PU(X)$ maps $x$ to $\uparw \{x\}$,
\item
The multiplication $\mu^\PU_X \colon \PU(\PU(X))\to \PU(X)$ maps $\calK$ to $\bigcup_{K\in\cal K} K$.
\end{itemize}

\noindent
Note that ${(\eta^\PU_X)}^{-1}(\Box U) = U$ and ${(\mu^\PU_X)}^{-1}(\Box U) = \Box\Box U$ and ${(\PU(f))}^{-1}(\Box U) = \Box f^{-1}(U)$ for every open $U\subseteq X$ and continuous $f\colon Y \to X$.

The proof of the next proposition was inspired by work by M.~Schr\"{o}der on the Scott topology on the open set lattice of sequential spaces (see, e.g.,~\cite{Sch:HMT}).

\begin{prop}\label{prop:UpperViet=Scott}
If $X$ is sober and countably based, then the upper Vietoris topology on $\PU(X)$ coincides with the Scott topology (where $\PU(X)$ is ordered by reverse subset inclusion).
\end{prop}
\begin{proof}
It is known that if $X$ is sober then $\PU(X)$ is sober with respect to the upper Vietoris topology (see Lemma 7.20 in~\cite{schalk:phd}). It follows that the upper Vietoris topology is contained within the Scott topology on $\PU(X)$.

Conversely, let $\calH\subseteq \PU(X)$ be any non-empty Scott-open set and fix $K \in \calH$. Assume for a contradiction that for every subset $\calU\subseteq \PU(X)$ which is open with respect to the upper Vietoris topology, if $K\in \calU$ then $\calU \not\subseteq \calH$. $\PU(X)$ is countably based because $X$ has a countable basis, so there exists a decreasing sequence of upper Vietoris open sets ${(\calU_n)}_{n\in\omega}$ forming a neighborhood basis for $K$ in the upper Vietoris topology. By assumption there exists $K_n \in \calU_n\setminus \calH$, hence the sequence ${(K_n)}_{n\in\omega}$ converges to $K$ with respect to the upper Vietoris topology, but none of the $K_n$ are in $\calH$.

For each $k\in\omega$, define $W_k = \bigcup_{n\geq k} K_n \cup K$. Since the $K_n$ converge to $K$ with respect to the upper Vietoris topology, any covering of $K$ by open sets will cover all but finitely many of the $K_n$, so it easily follows that each $W_k$ is a compact saturated subset of $X$. We claim that $K=\bigcap_{k\in\omega} W_k$. Clearly $K$ is contained in each of the $W_k$. Conversely, if $x\in X\setminus K$, then $U = X\setminus \dnarw \{x\}$ is an open subset of $X$ containing $K$, hence all but finitely many $K_n$ are in $\Box U$. Therefore, $x$ is in at most finitely many $W_k$, which implies $x \not\in \bigcap_{k\in\omega}W_k$.


Since the ordering on $\PU(X)$ is reverse subset inclusion, the $W_k$ are an increasing chain in $\PU(X)$ and $K$ is the least upper bound of this chain. Since $\calH$ is Scott-open, there must be some $k_0$ such that $W_k$ is in $\calH$ for all $k \geq k_0$. But then $K_n$ is in $\calH$ for every $n \geq k_0$, a contradiction.

Therefore, there must be $\calU\subseteq \PU(X)$ which is open with respect to the upper Vietoris topology and satisfies $K \in \calU \subseteq \calH$. It follows that the upper Vietoris and Scott topologies coincide on $\PU(X)$.
\end{proof}

\begin{exa}
The topology on $\PU(X)$ is not the weak topology in general, even when $X$ is quasi-Polish. As a counterexample, consider the natural numbers $\omega$ with the discrete topology. Then $\PU(\omega)$ consists of all finite subsets of $\omega$, including the empty set, ordered by reverse subset inclusion. The singleton $\{\emptyset\} = \Box\emptyset$ is open with respect to the upper Vietoris topology, but is not open with respect to the weak topology. Indeed, the subbasic open subsets for its weak topology are of the form $\calU_K = \{ K' \in \PU(X) \mid K'\not\supseteq K\}$ for $K\in\PU(X)$, and clearly $\emptyset \in \calU_K$ if and only if $K\not=\emptyset$. But if $m$ is any element of $K$, then $\{n\} \in \calU_K$ for each $n\not=m$. Therefore, any open neighborhood of $\emptyset$ in the weak topology contains cofinitely many compact saturated singletons.
\qed%
\end{exa}

We include the following for completeness. Note that sobriety is not required for the following.

\begin{prop}
If $X$ is a countably based $T_0$-space then $\eta^\PU_X(X) \in \bpi 2(\PU(X))$.
\end{prop}
\begin{proof}
Let $\calB$ be a countable basis for $X$ which is closed under finite unions. Each $K\in \PU(X)$ is equal to the saturation of a compact $T_1$-subspace of $X$ (see, e.g., Lemma~5.2 in~\cite{dBSS15b}), which implies that $K=\eta^\PU_X(x)$ for some $x\in X$ if and only if $K$ is non-empty (i.e. $K\not\in \Box \emptyset$) and
\[(\forall U,V\in\cal B) K \in \Box( U \cup V) \implies K \in (\Box U \cup \Box V).\]
Clearly this defines a $\bpi 2$-subset of $\PU(X)$.
\end{proof}

Using the same counterexample we gave for $\PL$, it is easy to see that $\PU$ does not preserve equalizers in general. However, we have the following, which was proven in~\cite{dBSS15b}.

\begin{prop}\label{prop:upper_preserves_pi2}
Assume $X$ and $Y$ are countably based $T_0$-spaces and $e\colon X \to Y$ is a topological embedding of $X$ into $Y$ as a $\bpi 2$-subspace. Then $\PU(e)\colon \PU(X) \to \PU(Y)$ is a topological embedding of $\PU(X)$ into $\PU(Y)$ as a $\bpi 2$-subspace.
\end{prop}
\begin{proof}
(\emph{Sketch}). Assume $Y$ is a countably based space, and let $\calB$ be a countable basis for $Y$ which is closed under finite unions. Assume $X\in \bpi 2(Y)$, and let $U_i,V_i$ $(i\in\omega)$ be open subsets of $Y$ such that $x\in X$ if and only if $(\forall i\in\omega)[x\in U_i \Rightarrow x\in V_i]$.

Let $e\colon X\to Y$ be the embedding of $X$ into $Y$, and let $f = \PU(e) \colon \PU(X) \to \PU(Y)$. By definition, $f$ maps $K\in \PU(X)$ to $\uparw K \in\PU(Y)$. It is easy to see that $f$ is a topological embedding of $\PU(X)$ into $\PU(Y)$.

Let $S$ be the subset of all $K\in \PU(Y)$ satisfying $K\in\Box(B \cup U_i) \Rightarrow K\in \Box(B \cup V_i)$ for all $i\in\omega$ and $B\in\cal B$. Clearly $S\in\bpi 2(\PU(Y))$. The theorem is then proved by showing that $S = range(f)$.
We refer to the original paper~\cite{dBSS15b} for a full proof.
\end{proof}

\begin{thm}\label{thrm:upperpowerspace}
If $X$ is quasi-Polish then $\PU(X)$ is quasi-Polish.
\end{thm}
\begin{proof}
The proof is similar to that of Theorem~\ref{thrm:lowerpowerspace}. Note that $\PU(D)$ is an $\omega$-algebraic domain whenever $D$ is (see M.~B.~Smyth~\cite{sm83} and A.~Schalk~\cite{schalk:phd}).
\end{proof}

\section{Convex (Vietoris) powerspaces}\label{sec:convexpowerspace}


The \emph{convex} (or \emph{Vietoris}) \emph{powerspace} $\PV(X)$ is defined to be the set of \emph{lenses} in $X$ with the Vietoris topology. A subset $L\subseteq X$ is a \emph{lens} if and only if there are $A\in\PL(X)$ and $K\in\PU(X)$ such that $L=A\cap K$. The Vietoris topology is the join of the lower and upper Vietoris topologies, and is generated by sets of the form $\Diamond U = \{ L\in\PV(X) \,|\, L\cap U \not=\emptyset\}$ and $\Box U = \{ L\in\PV(X) \,|\, L\subseteq U\}$, where $U$ varies over open subsets of $X$.

\begin{thm}\label{thrm:convexpowerspace}
If $X$ is quasi-Polish then $\PV(X)$ is quasi-Polish.
\end{thm}
\begin{proof}
The space $\PL(X)\times\PU(X)$ is quasi-Polish by Theorems~\ref{thrm:lowerpowerspace} and~\ref{thrm:upperpowerspace} and because quasi-Polish spaces are closed under topological products. Furthermore, it is easy to see that $\PV(X)$ is homeomorphic to the subspace of $\PL(X)\times\PU(X)$ consisting of all pairs $\langle A, K \rangle$ satisfying $Cl_X(A\cap K) = A$ and $\uparw (A\cap K) = K$. We can therefore identify $\PV(X)$ with this subspace of $\PL(X)\times\PU(X)$ and it only remains to show that it is a $\bpi 2$-subspace.

Let $\calB$ be a countable basis for $X$ which is closed under finite unions (hence $\calB$ contains the empty set as the empty union). We define $\calS$ to be the subset of $\PL(X)\times\PU(X)$ of all pairs $\langle A, K\rangle$ satisfying the following two conditions for each $U, V\in \calB$:
\begin{enumerate}
\item
If $A \in \Diamond U$ and $K\in \Box V$ then $A\in \Diamond (U \cap V)$, and
\item
If $K \in \Box(U \cup V)$ then either $A\in \Diamond U$ or else $K \in \Box V$.
\end{enumerate}
Since there are only countably many conditions, and each condition is an implication between open subsets (predicates), it is clear that $\calS$ is a $\bpi 2$-subset of $\PL(X)\times\PU(X)$. We conclude by proving that $\PV(X)$ is homeomorphic to $\calS$.

First, let $\langle A, K \rangle$ be any pair in $\PL(X)\times\PU(X)$ satisfying $Cl_X(A\cap K) = A$ and $\uparw (A\cap K) = K$. Fix any $U, V\in \calB$.
\begin{enumerate}
\item
If $A \in \Diamond U$ then $(A\cap K)\cap U \not=\emptyset$ because $Cl_X(A\cap K) = A$. If furthermore $K\in \Box V$ then $(A\cap K) \subseteq V$, hence $(A\cap K)\cap (U\cap V) \not=\emptyset$. Therefore, $A \in \Diamond U$ and $K\in \Box V$ implies $A\in \Diamond(U\cap V)$.
\item
If $K \in \Box(U \cup V)$ and $A\not\in \Diamond U$, then clearly $(A\cap K)\subseteq V$, hence $K\in\Box V$ because $\uparw (A\cap K) = K$.
\end{enumerate}
It follows that $\langle A, K \rangle\in \calS$.

Conversely, let $\langle A, K \rangle$ be any element of $\calS$.
\begin{enumerate}
\item
We prove that $K$ is dense in $A$, hence $Cl_X(A\cap K) = A$. Since $K$ is a compact saturated set, the assumption that $\calB$ is closed under finite unions implies the set $\calD = \{ V \in\calB \mid K \subseteq V\}$ satisfies $K=\bigcap\calD$. For each $V\in\calD$ we clearly have $K\in\Box V$, and if $U\in\calB$ is any basic open satisfying $A\in\Diamond U$, then the assumption $\langle A,K\rangle\in \calS$ implies $A\in\Diamond(U\cap V)$. Therefore, each $V\in\calD$ is dense in $A$. The countability of $\calB$ implies $\calD$ is countable, hence $K=\bigcap \calD$ is dense in $A$ because $X$ is completely Baire.
\item
If $W$ is any open set containing $A\cap K$, then the compactness of $K$ and the assumption that $\calB$ is closed under finite unions implies there exist $U,V\in\calB$ such that $U\subseteq X\setminus A$ and $V\subseteq W$ and $K\subseteq (U\cup V)$. Since $\langle A,K\rangle\in \calS$ and $A\not\in \Diamond U$, we must have $K\in\Box V$. Thus, every open set $W$ containing $A\cap K$ contains $K$, hence $\uparw (A\cap K) = K$.
\end{enumerate}
It follows that $\PV(X)$ is homeomorphic to $\calS$, hence $\PV(X)$ is quasi-Polish.
\end{proof}

The two sets of conditions used in the above proof to construct $\PV(X)$ as a $\bpi 2$-subset of $\PL(X)\times\PU(X)$ correspond to two of the axiom schemas in the locale theoretic presentation of the Vietoris locale as introduced by P.T.~Johnstone in~\cite{J85}. In particular, these are the two axiom schemas which integrate the $\Diamond$ and $\Box$ modalities.


\section{Double powerspaces and consonance}\label{sec:consonance}

In this section we will show that a topological property known as \emph{consonance} (see~\cite{dgl95,NS96,cw98,bo99}) is equivalent to the commutativity of the upper and lower powerspaces under a naturally defined homeomorphism. Our characterization was inspired by A. Bouziad's characterization of consonant spaces in~\cite{bo99}. In particular, our Theorem~\ref{thrm:consonance} is closely related to Theorem~2 in~\cite{bo99}.



\begin{lem}\label{lem:subbasisForAK_KA}
Let $X$ be a topological space.
\begin{enumerate}
\item\label{lem:item:AK}
The topology on $\PL(\PU(X))$ is generated by sets of the form
\[\Diamond\Box U = \{ \calA\in\PL(\PU(X)) \,|\, (\exists K\in \calA)\, K\subseteq U\},\]
where $U\subseteq X$ is open.
\item\label{lem:item:KA}
The topology on $\PU(\PL(X))$ is generated by sets of the form
\[\Box\Diamond U = \{ \calK\in\PU(\PL(X)) \,|\, (\forall A\in \calK)\, A\cap U\not=\emptyset\},\]
where $U\subseteq X$ is open.
\end{enumerate}
\end{lem}
\begin{proof}

In the following, $\{ U^i_j \,|\,  i\in I \mbox{ and } j\in J_i \}$ vary over doubly indexed families of open subsets of $X$ with $J_i$ a finite index set for each $i\in I$.

Part (\ref{lem:item:AK}) easily follows from the fact that the topology of $\PL(\PU(X))$ is generated by sets of the form
\[\Diamond \calU = \Diamond\bigcup_{i\in I}\bigcap_{j\in J_i} \Box U^i_j = \bigcup_{i\in I}\Diamond\Box \bigcap_{j\in J_i}U^i_j,\]
where the second equality holds because $\Diamond$ commutes with unions and $\Box$ commutes with finite intersections.

For Part (\ref{lem:item:KA}), first note that the topology of $\PU(\PL(X))$ is generated by sets of the form
\[\Box \calU = \Box\bigcup_{i\in I}\bigcap_{j\in J_i} \Diamond U^i_j.\]
For any $\calK\in \Box\cal U$, the compactness of $\calK$ implies there is finite $F\subseteq I$ such that
\[\calK\in \Box\bigcup_{i\in F}\bigcap_{j\in J_i} \Diamond U^i_j.\]
Let $S$ be the set of all selection functions $s$ that map each $i\in F$ to some $s(i) \in J_i$. Then
\[ \Box\bigcup_{i\in F}\bigcap_{j\in J_i} \Diamond U^i_j = \Box\bigcap_{s\in S} \bigcup_{i\in F} \Diamond U^i_{s(i)}\]
by distributivity. Using again the fact that $\Diamond$ commutes with unions and $\Box$ with finite intersections, we obtain
\[\calK \in \bigcap_{s\in S} \Box\Diamond\bigcup_{i\in F} U^i_{s(i)} \subseteq \Box \calU,\]
which completes the proof.
\end{proof}

\begin{prop}
The topology on $\PU(\PL(X))$ coincides with the weak-topology.
\end{prop}
\begin{proof}
Let $U\subseteq X$ be open. Then $\uparw \{(X\setminus U)\}$ is in $\PU(\PL(X))$ hence the complement of $\dnarw\{\uparw \{(X\setminus U)\}\}$ is open with respect to the weak topology on $\PU(\PL(X))$. We show that $\Box \Diamond U$ is equal to the complement of $\dnarw\{\uparw \{(X\setminus U)\}\}$.

For $\calK\in\PU(\PL(X))$ we have $\calK \not\in  \Box \Diamond U$ if and only if  there exists $A\in \calK$ such that $A\subseteq (X\setminus U)$ if and only if $(X\setminus U) \in \calK$ (because $\calK$ is saturated) if and only if $\uparw \{(X\setminus U)\} \subseteq \calK$ if and only if $\calK \in \dnarw\{\uparw \{(X\setminus U)\}\}$ (because the specialization order of $\PU(\PL(X))$ is reverse subset inclusion).
\end{proof}

\begin{defi}\label{def:AK_KA}
For each topological space $X$ define $\akka_X\colon \PL(\PU(X)) \to \PU(\PL(X))$ and\\ $\kaak_X\colon \PU(\PL(X)) \to \PL(\PU(X))$ as
\begin{align*}
\akka_X(\calA) &= \{ A\in \PL(X) \,|\, (\forall K\in \calA)\, A\cap K\not=\emptyset\},\\
\kaak_X(\calK) &= \{ K\in \PU(X) \,|\, (\forall A\in \calK)\, A\cap K\not=\emptyset\}.
\tag*{\qed}
\end{align*}
\end{defi}
We will usually omit the subscripts from $\akka_X$ and $\kaak_X$ when the space is clear from the context. R.~Heckmann used similarly defined functions in~\cite{H92} to prove a commutativity result for the lower powerdomain and a modified version of the upper powerdomain.

\begin{lem}\label{lem:akka_correct}
The function $\akka$ is well-defined for each topological space $X$. Furthermore, $\akka^{-1}(\Box\Diamond U) = \Diamond \Box U$ for each open $U\subseteq X$. In particular, $\akka$ is a topological embedding.
\end{lem}
\begin{proof}
We first show that $\akka$ is well-defined. Fix $\calA\in\PL(\PU(X))$. By the Alexander subbase theorem it suffices to show that every cover $\bigcup_{i\in I} \Diamond U_i$ of $\akka(\calA)$ admits a finite subcover. Let $U=\bigcup_{i\in I} U_i$. Every $A\in \akka(\calA)$ has non-empty intersection with $U$, hence $X\setminus U$ is not in $\akka(\calA)$. By the definition of $\akka$ this implies there is $K\in \calA$ such that $K\subseteq U$. As $K$ is compact, there is finite $F\subseteq I$ such that $K\subseteq \bigcup_{i\in F} U_i$. Each $A\in \akka(\calA)$ intersects $K$, hence each $A\in \akka(\calA)$ intersects $U_i$ for some $i\in F$. Therefore, $\akka(\calA)\subseteq \bigcup_{i\in F} \Diamond U_i$. This proves that $\akka(\calA)\in \PU(\PL(X))$, hence $\akka$ is well-defined.

If $\calA \in \Diamond\Box U$ then there is $K\in\calA$ such that $K\subseteq U$. Each $A\in \akka(\calA)$ intersects $K$ hence $A\cap U\not=\emptyset$, which implies $\akka(\calA)\in \Box\Diamond U$.

Conversely, if $\akka(\calA)\in \Box\Diamond U$ then every $A\in \akka(\calA)$ has non-empty intersection with $U$, hence $X\setminus U \not\in \akka(\calA)$. This implies there exists $K\in\calA$ such that $(X\setminus U)\cap K =\emptyset$ thus $K\subseteq U$. Therefore, $\calA \in \Diamond\Box U$.

It immediately follows that $\akka$ is continuous. Furthermore, Lemma~\ref{lem:subbasisForAK_KA} and the fact that $\akka^{-1}(\Box\Diamond U) = \Diamond \Box U$ for each open $U\subseteq X$ implies that every open subset of $\PL(\PU(X))$ is equal to the preimage under $\akka$ of some open subset of $\PU(\PL(X))$. Therefore, $\akka$ is a topological embedding.
\end{proof}

\begin{lem}\label{lem:kaak_welldefined}
The function $\kaak$ is well-defined. Furthermore, $\kaak^{-1}(\Diamond\Box U)\subseteq \Box\Diamond U$ for each open $U\subseteq X$.
\end{lem}
\begin{proof}
Fix $\calK\in\PU(\PL(X))$. If $K\not\in \kaak(\calK)$ then there is $A\in \calK$ such that $A\cap K=\emptyset$. Clearly $U=X\setminus A$ is an open subset of $X$, hence $\Box U$ is an open subset of $\PU(X)$ which contains $K$ and is disjoint from $\kaak(\calK)$. Therefore, $\kaak(\calK)\in\PL(\PU(X))$.

If $\kaak(\calK)\in\Diamond\Box U$ then there is $K\in \kaak(\calK)$ with $K\subseteq U$. Every $A\in\calK$ intersects $K$ by the definition of $\kaak$, which implies every $A\in\calK$ intersects $U$. Therefore, $\calK \in \Box\Diamond U$.
\end{proof}

Unfortunately, the previous lemma does not guarantee that $\kaak$ is continuous. Our next goal is to characterize the spaces for which $\kaak$ is continuous.

Given a topological space $X$, we let $\PO(X)$ denote the lattice of open subsets of $X$ ordered by inclusion and equipped with the Scott-topology. $\PO$ determines a contravariant endofunctor on the category of topological spaces by mapping a continuous function $f\colon X\to Y$ to the corresponding frame homomorphism $\PO(f) = f^{-1} \colon \PO(Y)\to\PO(X)$, which is clearly Scott-continuous.

For $K\in\PU(X)$ we define
\[\triangledown K = \{U\in\PO(X) \mid K\subseteq U\},\]
which is known to be a Scott-open filter in $\PO(X)$.

\begin{defi}[See~\cite{dgl95}]
A space $X$ is \emph{consonant} if and only if for every Scott-open $\calH \subseteq \PO(X)$ and every $U\in\calH$ there exists $K\in\PU(X)$ such that $U \in\triangledown K \subseteq \calH$.
\qed%
\end{defi}

By the Hoffmann--Mislove theorem (see~\cite{g03}), if $X$ is sober then $X$ is consonant if and only if the Scott topology on $\PO(X)$ has a basis consisting of Scott-open filters. It was shown in~\cite{dBSS15b} that quasi-Polish spaces are consonant, and it is known that a separable co-analytic metrizable space is consonant if and only if it is Polish~\cite{bo99}. In particular, the space of rationals with the subspace topology inherited from the reals is not consonant~\cite{cw98}. 

\begin{defi}\label{def:KAOO}
For each topological space $X$ define $\kaoo_X\colon \PU(\PL(X)) \to \PO(\PO(X))$ and\\ $\ooka_X\colon \PO(\PO(X)) \to \PU(\PL(X))$ as
\begin{align*}
\kaoo_X(\calK) =& \{U\in\PO(X) \,|\, \calK\in\Box\Diamond U\},\\
\ooka_X(\calH) =& \bigcap_{U\in\calH}\Diamond U.
\tag*{\qed}
\end{align*}
\end{defi}
We will usually omit the subscripts from $\kaoo_X$ and $\ooka_X$ when the space is clear from the context.


\begin{lem}\label{lem:KAX->OOX}
The function $\kaoo$ is well-defined for each topological space $X$. 
\end{lem}
\begin{proof}
Fix $\calI\subseteq\PO(X)$ and assume $U = \bigcup_{V \in \calI} V$ is in $\kaoo(\calK)$. Then $\calK \subseteq \Diamond U = \bigcup_{V\in\calI} \Diamond V$. Since $\calK$ is compact there is finite $\calF\subseteq \calI$ such that $\calK\subseteq \bigcup_{V\in \calF} \Diamond V = \Diamond \bigcup_{V\in \calF} V$. Therefore, $\bigcup_{V\in\calF} V$ is in $\kaoo(\calK)$, hence $\kaoo(\calK)$ is Scott-open.
\end{proof}



\begin{lem}\label{lem:OOX->KAX}
The function $\ooka$ is well-defined for each topological space $X$. Furthermore, $U\in\calH$ if and only if $\ooka(\calH)\in\Box\Diamond U$ for each $\calH\in\PO(\PO(X))$.
\end{lem}
\begin{proof}
We first prove the second claim. Clearly, if $U\in\calH$ then $\ooka(\calH)\subseteq \Diamond U$. Conversely, $\ooka(\calH) \subseteq\Diamond U$ implies $(X\setminus U) \not\in\ooka(\calH)$ hence there is $V\in\calH$ such that $(X\setminus U)\not\in \Diamond V$. Therefore, $V\subseteq U$ which implies $U\in\calH$ because $\calH$ is an upper set.

Next we show that $\ooka(\calH)$ is compact. Fix $\calI\subseteq\PO(X)$ and assume $\ooka(\calH)\subseteq \bigcup_{V\in\calI}\Diamond V$. By setting $U=\bigcup_{V\in\calI}V$ we obtain $\ooka(\calH) \in \Box\Diamond U$ hence $U\in \calH$. Since $\calH$ is Scott-open there is finite $\calF\subseteq\calI$ such that $\bigcup_{V\in\calF}V \in \calH$. It follows that every $A\in\ooka(\calH)$ has non-empty intersection with some $V\in\calF$ hence $\ooka(\calH)\subseteq \bigcup_{V\in\calF}\Diamond V$.
\end{proof}




A locale theoretic version of the following was first shown by S.~Vickers~\cite{V04} and S.~Vickers and C.~Townsend~\cite{VT04}. Recall that $\PU(X)$ is ordered by reverse subset inclusion.

\begin{thm}\label{thrm:KA_OO_iso}
$\PU(\PL(X))$ and $\PO(\PO(X))$ are isomorphic lattices (via $\kaoo$ and its inverse $\ooka$) for every topological space $X$.
\end{thm}
\begin{proof}
It is easy to see that $\kaoo\colon\PU(\PL(X))\to\PO(\PO(X))$ and $\ooka\colon \PO(\PO(X)) \to \PU(\PL(X))$ are order-preserving. For $\calH\in\PO(\PO(X))$ and $U\in\PO(X)$ we have
\begin{eqnarray*}
U\in\kaoo(\ooka(\calH)) &\iff& \ooka(\calH)\in\Box\Diamond U\,\text{ (by definition of $\kaoo$)}\\
&\iff& U\in\calH \,\text{ (by Lemma~\ref{lem:OOX->KAX})}.
\end{eqnarray*}
Furthermore, for $\calK\in\PU(\PL(X))$ we have
\begin{eqnarray*}
\ooka(\kaoo(\calK)) \in\Box\Diamond U &\iff& U\in\kaoo(\calK)\,\text{ (by Lemma~\ref{lem:OOX->KAX})}\\
&\iff& \calK\in\Box\Diamond U\,\text{ (by definition of $\kaoo$)},
\end{eqnarray*}
thus Lemma~\ref{lem:subbasisForAK_KA}(\ref{lem:item:KA}) implies $\ooka\circ\kaoo$ is the identity on $\PU(\PL(X))$. Therefore, $\kaoo$ and $\ooka$ are inverses of each other. 
\end{proof}

We can say more when $X$ is countably based.

\begin{thm}\label{thrm:scott_top_basis_for_OOX}
If $X$ is a countably based space then $\PU(\PL(X))$ with the upper Vietoris topology is homeomorphic (via $\kaoo$ and its inverse $\ooka$) to $\PO(\PO(X))$ with the Scott topology. The Scott topology on $\PO(\PO(X))$ has a subbasis given by sets of the form
\[\boxtimes U = \{\calH\in\PO(\PO(X)) \mid U\in \calH\}\]
for open $U\subseteq X$.
\end{thm}
\begin{proof}
Since $\PL(X)$ is always sober, Proposition~\ref{prop:UpperViet=Scott} implies the upper Vietoris topology and Scott topology coincide for $\PU(\PL(X))$. Theorem~\ref{thrm:KA_OO_iso} then implies $\PU(\PL(X))$ and $\PO(\PO(X))$ are homeomorphic. The topology on $\PU(\PL(X))$ is generated by sets of the form $\Box\Diamond U$, hence the topology on $\PO(\PO(X))$ is generated by sets of the form $\boxtimes U = \ooka^{-1}(\Box\Diamond U)$.
\end{proof}

\begin{lem}\label{lem:kaoo_and_kaak}
Let $X$ be a topological space. Then $\triangledown K \subseteq \kaoo(\calK)$ if and only if $K \in \kaak(\calK)$ for each $\calK\in\PU(\PL(X))$ and $K\in\PU(X)$.
\end{lem}
\begin{proof}
Assume $\triangledown K \subseteq \kaoo(\calK)$. If $A\in\calK$ then $\calK\not\subseteq \Diamond(X\setminus A)$ hence $(X\setminus A)\not\in \kaoo(\calK)$. Thus $K\not\subseteq  (X\setminus A)$ which implies $K\cap A \not=\emptyset$. Therefore, $K\in\kaak(\calK)$.

Conversely, assume $K\in\kaak(\calK)$ and $K\subseteq U$. Each $A\in\calK$ intersects $K$ hence $A\cap U\not=\emptyset$, which implies $\calK\in\Box\Diamond U$. Therefore, $U\in\kaoo(\calK)$.
\end{proof}

The next result is inspired by A.~Bouziad's Theorem~2 in~\cite{bo99}.

\begin{thm}\label{thrm:consonance}
The following are equivalent for every topological space $X$:
\begin{enumerate}
\item\label{enum:cons}
$X$ is consonant,
\item\label{enum:akka_surj}
$\akka\colon \PL(\PU(X)) \to \PU(\PL(X))$ is a bijection,
\item\label{enum:kaak_eq}
$\kaak^{-1}(\Diamond\Box U) = \Box \Diamond U$ for each open $U\subseteq X$.
\end{enumerate}
\end{thm}
\begin{proof}
(\ref{enum:cons} $\Rightarrow$~\ref{enum:akka_surj}). Assume $X$ is consonant. We already showed that $\akka$ is injective in Lemma~\ref{lem:akka_correct}, so it only remains to show that $\akka$ is a surjection. Fix $\calK\in\PU(\PL(X))$. For each $U\in\PO(X)$ we have
\begin{eqnarray*}
\calK \in\Box\Diamond U &\iff& U\in\kaoo(\calK)\\
&\iff& (\exists K\in\PU(X))[U\in \triangledown K \subseteq \kaoo(\calK)\,]\text{ (by consonance)}\\
&\iff& (\exists K\in\kaak(\calK))\, K\subseteq U\,\text{ (by Lemma~\ref{lem:kaoo_and_kaak})}\\
&\iff& \kaak(\calK) \in \Diamond\Box U\\
&\iff& \akka(\kaak(\calK)) \in \Box\Diamond U\,\text{ (by Lemma~\ref{lem:akka_correct})},
\end{eqnarray*}
hence $\calK = \akka(\kaak(\calK))$. Therefore, $\akka$ is surjective.

(\ref{enum:akka_surj} $\Rightarrow$~\ref{enum:kaak_eq}). Fix any $\calK \in \Box\Diamond U$ and let $\calA\in\PL(\PU(X))$ be such that $\akka(\calA)=\calK$. Lemma~\ref{lem:akka_correct} implies $\calA \in \Diamond\Box U$, so there is some $K\in\calA$ such that $K\subseteq U$. Since $\akka(\calA)=\calK$, every $A\in\calK$ has non-empty intersection with $K$, which implies $K\in \kaak(\calK)$. Clearly $K\in \kaak(\calK)$ demonstrates that $\kaak(\calK)\in\Diamond\Box U$. Therefore, $\Box \Diamond U \subseteq \kaak^{-1}(\Diamond\Box U)$, and the converse inclusion is by Lemma~\ref{lem:kaak_welldefined}.

(\ref{enum:kaak_eq} $\Rightarrow$~\ref{enum:cons}). If $\calH\subseteq\PO(X)$ is Scott-open and $U\in\calH$ then $\ooka(\calH)\in \Box\Diamond U$ by Lemma~\ref{lem:OOX->KAX}. The assumption on $\kaak$ implies $\kaak(\ooka(\calH)) \in \Diamond\Box U$, hence there is $K\in \kaak(\ooka(\calH))$ with $K\subseteq U$. By the definition of $\kaak$, every $A\in\ooka(\calH)$ has non-empty intersection with $K$. It follows that whenever $V\subseteq X$ is open and $K\subseteq V$ then $\ooka(\calH)\in\Box\Diamond V$, hence $V\in \calH$ by Lemma~\ref{lem:OOX->KAX}. Therefore, $X$ is consonant.
\end{proof}




Lemma~\ref{lem:akka_correct} and Theorem~\ref{thrm:consonance} together imply the following.

\begin{thm}\label{thrm:cons_AK_KA_homeo}
If $X$ is consonant then $\PL(\PU(X))$ and $\PU(\PL(X))$ are homeomorphic (via $\akka$ and its inverse $\kaak$).
\qed%
\end{thm}

\begin{cor}\label{cor:cntblybasedcons_AK_KA_OO}
If $X$ is a countably based consonant space, then $\PL(\PU(X))$, $\PU(\PL(X))$ and $\PO(\PO(X))$ are all homeomorphic. Their topologies coincide with both the weak topology and the Scott topology.
\qed%
\end{cor}

In particular, if $X$ is quasi-Polish, then $\PL(\PU(X))$, $\PU(\PL(X))$ and $\PO(\PO(X))$ are all naturally homeomorphic and quasi-Polish. The naturality of the transformations $\akka$, $\kaak$, $\kaoo$, and $\ooka$ is easily verified by noting how the modality operators on basic open sets change under preimages of the relevant maps. Similarly, it can be shown that $\akka$ and $\kaak$ are distributive laws (in the sense of Beck~\cite{be69}) between the monads $\PL$ and $\PU$ on the category of quasi-Polish spaces.  The locale theoretic version of this observation has already been made by S.~Vickers in~\cite{V04}. Currently we do not know if $\PL$ and $\PU$ preserve consonance, but if $\PL$ and $\PU$ are well-defined monads on the category of consonant spaces then $\akka$ and $\kaak$ are distributive laws in this more general setting as well.

It is important to note that $\PO(X)$ will not be quasi-Polish in general, as we will show in the following proposition. Parts of the proof depend on results we will obtain in later sections, but it is useful to give the proposition now so that the reader can better appreciate the motivation of those sections. The equivalences as described below depend on the sobriety of $X$, and are a special case of the more general observation by Matthias Schr\"{o}der (personal communication) that if $X$ is countably based then $\PO(X)$ (with the Scott-topology) is countably based if and only if $X$ is core-compact.

\begin{prop}
The following are equivalent for every quasi-Polish space $X$:
\begin{enumerate}
\item\label{prop:enum:OXcntbased1}
$X$ is locally compact,
\item\label{prop:enum:OXcntbased2}
$\PO(X)$ is quasi-Polish,
\item\label{prop:enum:OXcntbased3}
$\PO(X)$ is countably based,
\item\label{prop:enum:OXcntbased4}
$\PO(X)$ is locally compact,
\item\label{prop:enum:OXcntbased5}
$\PO(\PO(X))$ is locally compact.
\end{enumerate}
\end{prop}
\begin{proof}
(\ref{prop:enum:OXcntbased1} $\Rightarrow$~\ref{prop:enum:OXcntbased2},~\ref{prop:enum:OXcntbased3},~\ref{prop:enum:OXcntbased4},~\ref{prop:enum:OXcntbased5}). If  $X$ is locally compact then $\PO(X)$ is a continuous lattice, hence locally compact with respect to the Scott-topology (see~\cite{g03}). Furthermore, since $X$ is also countably based, there exists countable $\calB\subseteq \PU(X)$ such that whenever $x \in U$ there is $K\in \calB$ with $x \in Int_X(K)\subseteq K\subseteq U$, where $Int_X(\cdot)$ is the interior operator for $X$. Since $X$ is consonant, one easily verifies that $\{ \triangledown K \mid K\in\calB\}$ is a countable basis for $\PO(X)$. We will see later (Corollary~\ref{cor:OX_sober}) that $\PO(X)$ is sober, and every countably based locally compact sober space is quasi-Polish~\cite{dbr}. Repeating the above argument for $\PO(X)$ shows that $\PO(\PO(X))$ is locally compact.

(\ref{prop:enum:OXcntbased2} $\Rightarrow$~\ref{prop:enum:OXcntbased3}) holds by definition.

(\ref{prop:enum:OXcntbased3} $\Rightarrow$~\ref{prop:enum:OXcntbased1}). Assume $\PO(X)$ is countably based. Both $X$ and $\PO(X)$ are sequential spaces, hence $E=\{ (x, U) \mid x\in U\}$ is a sequentially open subset of $X \times \PO(X)$ (see items (1) and (3) of Proposition~2.2 in~\cite{Sch:HMT}). Therefore, $E$ is open because $X \times \PO(X)$ is countably based. Now using the fact that $X$ is consonant, given any pair $(x,U)\in X\times\PO(X)$, we have that $x\in U$ if and only if there exists $W\in \PO(X)$ and $K\in\PU(X)$ with $(x,U) \in W\times \triangledown K \subseteq E$, which is possible if and only if $x \in W \subseteq K \subseteq U$. Therefore, $X$ is locally compact.

(\ref{prop:enum:OXcntbased4} $\Rightarrow$~\ref{prop:enum:OXcntbased1}). Assume $\PO(X)$ is locally compact. If $U$ is an open neighborhood of $x\in X$ then the local compactness of $\PO(X)$ and the consonance of $X$ implies there are $\calK\in\PU(\PO(X))$ and $K\in\PU(X)$ with $U\in\triangledown K \subseteq \calK \subseteq \triangledown\uparw\{x\}$. We will show later (Corollary~\ref{cor:compact_intersections}) that $V = \bigcap\calK$ is open, and it is easy to verify that $x \in V \subseteq K \subseteq U$. Therefore, $X$ is locally compact.

(\ref{prop:enum:OXcntbased5} $\Rightarrow$~\ref{prop:enum:OXcntbased3}). Assume $\PO(\PO(X))$ is locally compact. $\PO(\PO(X))$ is quasi-Polish by our assumption on $X$, so the implication (\ref{prop:enum:OXcntbased1} $\Rightarrow$~\ref{prop:enum:OXcntbased3}) applies and $\PO(\PO(\PO(X)))$ is countably based. Since $X$ is countably based it is a qcb-space,\footnote{A \emph{q}uotient of a \emph{c}ountably \emph{b}ased space (see~\cite{BSS07} for a survey).} hence Proposition~2.2(6) in~\cite{Sch:HMT} implies $\PO(X)$ is a sequential space, and it follows that $\PO(X)$ topologically embeds into $\PO(\PO(\PO(X)))$ by Corollary~6.4 of~\cite{dBSS15b}. Therefore, $\PO(X)$ is countably based.
\end{proof}

From the above proposition, we see that $\PO(\PO(\omega^{\omega}))$, where $\omega^{\omega}$ is the Baire space, is an example of a complete lattice which is quasi-Polish but not locally compact when given the Scott topology.



\section{Co-consonance and strongly compact sets}\label{sec:coconsonance}

In this section we gather together some results which will be useful when we later investigate the powerspace monads on $\PO(X)$.

For $A\in\PL(X)$ we define
\[\triangleup A = \{U\in\PO(X) \mid A\cap U \not=\emptyset\}.\]
Since $\triangleup A$ equals the complement of $\dnarw\{(X\setminus A)\}\in\PL(\PO(X))$, it is clear that $\triangleup A$ is a Scott-open subset of $\PO(X)$.

\begin{defi}
A space $X$ is \emph{co-consonant} if and only if for every Scott-open $\calH \subseteq \PO(X)$ and every $U\in\calH$ there exists finite $\calF \subseteq \PL(X)$ such that $U \in \bigcap_{A\in\calF}\triangleup A \subseteq \calH$.
\qed%
\end{defi}

Equivalently, $X$ is co-consonant if and only if the weak topology and Scott topology on $\PO(X)$ agree.

Thus, a space $X$ is consonant if the range of the mapping
\[\triangledown\colon \PU(X) \to \PO(\PO(X)),\quad K \mapsto \triangledown K\]
is a subbase for the Scott topology on $\PO(X)$, and $X$ is co-consonant if the range of the mapping
\[\triangleup \colon \PL(X) \to \PO(\PO(X)),\quad A \mapsto \triangleup A\]
is a subbase. Both of these mappings are continuous embeddings when $X$ is countably based.

\begin{prop}\label{prop:cntbsd_OX_cons_cocons}
If $X$ is a countably based space then $\PO(X)$ is both consonant and co-consonant.
\end{prop}
\begin{proof}
Let $\frakU\subseteq \PO(\PO(X))$ be Scott-open and assume $\calH\in\frakU$. Since $X$ is countably based, Theorem~\ref{thrm:scott_top_basis_for_OOX} implies there are $U_0,\ldots, U_n \in\PO(X)$ such that $\calH \in \boxtimes U_0 \cap \cdots \cap \boxtimes U_n \subseteq \frakU$. Then $\frakK = \uparw\{U_0,\ldots,U_n\}$ is in $\PU(\PO(X))$ and $\calH\in \triangledown\frakK\subseteq\frakU$, which shows that $\PO(X)$ is consonant. Furthermore, $\dnarw\{U_i\}\in\PL(\PO(X))$ and $\boxtimes U_i = \triangleup\dnarw\{U_i\}$, which shows that $\PO(X)$ is co-consonant.
\end{proof}




Next we give a characterization of co-consonant spaces in terms of their compact subsets.

\begin{defi}[R.~Heckmann~\cite{H92}]
A compact subset $K$ of a space $X$ is \emph{strongly compact} if and only if for every open $U\subseteq X$ with $K\subseteq U$, there is finite $F\subseteq X$ such that $K\subseteq \uparw F \subseteq U$.
\qed%
\end{defi}

\begin{prop}\label{prop:coconsonant_strongcompact}
If $X$ is co-consonant then every $K\in\PU(X)$ is strongly compact. The converse holds if $X$ is consonant.
\end{prop}
\begin{proof}
Assume $X$ is co-consonant and $K\in\PU(X)$. If $U\subseteq X$ is open and $K\subseteq U$, then $\triangledown K$ is a Scott-open subset of $\PO(X)$ containing $U$, hence by co-consonance there is finite $\calF \subseteq \PL(X)$ such that $U \in \bigcap_{A\in\calF}\triangleup A \subseteq \triangledown K$. Choose $x_A \in U\cap A$ for each $A\in\calF$ and let $F=\{x_A\mid A\in\calF\}$. For any $y\in K$ the set $V = X\setminus \dnarw \{y\}$ is an open set which does not contain $K$, hence $V\not\in \bigcap_{A\in\calF}\triangleup A$. It follows that there is some $A\in\calF$ such that $A\subseteq \dnarw \{y\}$, hence $x_A\leq y$. Therefore, $K\subseteq \uparw F\subseteq U$.

For the converse, assume $X$ is consonant and every compact subset of $X$ is strongly compact. If $U\subseteq X$ is open and $K\in\PU(X)$ is such that $U\in\triangledown K$, then by strong compactness of $K$ there is finite $F$ such that $K\subseteq \uparw F\subseteq U$. Then $U \in \bigcap_{x\in F}\triangleup (\dnarw \{x\}) \subseteq \triangledown K$, hence $X$ is co-consonant.
\end{proof}

In the following sections we will need to use the fact that $\PU(X)$ is co-consonant whenever $X$ is quasi-Polish. This will easily follow from the fact that every quasi-Polish space is a Wilker space, which is an important result in its own right.

\begin{defi}[P.~Wilker~\cite{W70}]
A space is a \emph{Wilker space} if and only if for every pair of opens $U_1$ and $U_2$ and compact set $K\subseteq U_1 \cup U_2$, there exist compact sets $K_1\subseteq U_1$ and $K_2\subseteq U_2$ such that $K\subseteq K_1\cup K_2$.
\qed%
\end{defi}

Every Hausdorff space and every locally compact space is a Wilker space (see the original paper~\cite{W70} for a proof and other applications, where this property is called condition $(D)$).

We now work towards proving that every quasi-Polish space is a Wilker space. The following definition is due to V.~Becher and S.~Grigorieff~\cite{BG15}. Note however that we use the opposite of the relations defined in that paper so that they will be compatible with subset inclusion. We also explicitly include uniqueness in the fourth item of the definition because we are only concerned with $T_0$-spaces.

\begin{defi}
Let $X$ be a topological space. A \emph{convergent approximation relation} for $X$ is a binary relation $\prec$ on some basis $\calB\subseteq\PO(X)$ such that for all $U,V,W\in\calB$:
\begin{enumerate}
\item
$U\prec V$ implies $U\subseteq V$,
\item
$U \prec V \subseteq W$ implies $U\prec W$,
\item
For each $x\in U$ there is $O\in\calB$ such that $x\in O \prec U$,
\item
Every sequence ${(U_i)}_{i\in\omega}$ in $\calB$ satisfying $(\forall i\in\omega)[\, U_{i+1}\prec U_i \,]$ is a neighborhood basis for a unique $x \in \bigcap_{i\in\omega}U_i$.
\end{enumerate}
A topological space $X$ is a \emph{convergent approximation space} if and only if there is a convergent approximation relation $\prec$ for $X$ on some basis $\calB$.
\qed%
\end{defi}

As shown in~\cite{BG15}, if $\calB$ and $\calD$ are bases for $X$ and $\prec$ is a convergent approximation relation on $\calB$, then we can obtain a convergent approximation relation $\prec'$ on $\calD$ by defining
\[O\prec' W \iff (\exists U, V\in\calB)[\, O\subseteq U \prec V \subseteq W\,]\]
for $O,W\in\calD$.

For simplicity, in this paper we will only consider convergent approximation relations on $X$ for $\calB=\PO(X)$.

\begin{prop}[V.~Becher and S.~Grigorieff~\cite{BG15}]
A countably based $T_0$-space is quasi-Polish if and only if it is a convergent approximation space.
\qed%
\end{prop}

Let $\omega^{<\omega}$ denote the set of finite sequences of natural numbers. We write $\varepsilon$ for the empty sequence. For $s\in \omega^{<\omega}$ and $n\in\omega$ we write $s\diamond n$ to denote the finite sequence obtained by appending $n$ to the end of $s$. A finite sequence of the form $s\diamond n$ is called an \emph{immediate successor} of $s$. For $s,t\in\omega^{<\omega}$ we write $s \sqsubseteq t$ to denote that $s$ is an initial prefix of $t$, and $s \sqsubset t$ if $s\sqsubseteq t$ and $s\not=t$. Similarly, we write $s\sqsubset p$ to denote that $s$ is a finite initial prefix of a countably infinite sequence $p\in\omega^\omega$.

A \emph{tree} is a subset $T\subseteq\omega^{<\omega}$ such that $s\sqsubseteq t\in T$ implies $s\in T$. An infinite sequence $p\in\omega^\omega$ is a \emph{path} in $T$ if $s\sqsubset p$ implies $s\in T$ for each $s\in\omega^{<\omega}$. We denote the set of paths in $T$ by $[T]$, and view $[T]$ as a topological space with the subspace topology inherited from $\omega^\omega$. Note that $[T]$ is a closed subspace of $\omega^{\omega}$.

\begin{defi}
Let $X$ be a quasi-Polish space with convergent approximation relation $\prec$. An \emph{approximation scheme} for $X$ is a function $f\colon T \to\PO(X)$, where $T$ is a tree, such that $s \sqsubset t$ implies $f(t) \prec f(s)$ for each $s,t\in T$. Define $\widehat{f}\colon [T]\to X$ to be the function mapping each infinite path $p\in [T]$ to the unique point having $\{ f(s) \mid s\sqsubset p\}$ as a neighborhood basis.
\qed%
\end{defi}

\begin{lem}
Let $X$ be a quasi-Polish space with convergent approximation relation $\prec$, and let $f\colon T \to\PO(X)$ be an approximation scheme. Then $\widehat{f}\colon [T]\to X$ is a continuous function.
\end{lem}
\begin{proof}
Assume $p\in[T]$ and $\widehat{f}(p)\in U$ is open. Then there exists some finite $s\sqsubset p$ such that $\widehat{f}(p) \in f(s) \prec U$. The open set $V_s = \{ q\in [T] \mid s\sqsubset q\}$ satisfies $p \in V_s \subseteq \widehat{f}^{-1}(U)$.
\end{proof}

A tree $T$ is \emph{finitely branching} if every $s\in T$ has at most finitely many immediate successors in $T$. It is well known that $[T]$ is a compact space whenever $T$ is finitely branching, which can be shown using K\"{o}nig's Lemma.

The following was observed during discussions between the first author and Klaus Keimel, and we are grateful to Klaus Keimel for allowing us to include the result here.

\begin{thm}
Every quasi-Polish space is a Wilker space.
\end{thm}
\begin{proof}
Let $X$ be quasi-Polish with convergent approximation relation $\prec$. Fix $K\in\PU(X)$ and $U_1,U_2\in\PO(X)$ and assume $K\subseteq U_1\cup U_2$.

We inductively define two sequences of finite sets $F_k,G_k\subseteq \omega^k$ and simultaneously construct functions $f\colon \bigcup_{k\in\omega} F_k\to \PO(X)$ and $g\colon \bigcup_{k\in\omega} G_k\to \PO(X)$ such that
\[K\subseteq \Bigl(\bigcup_{s\in F_k}f(s)\Bigr)\cup \Bigl(\bigcup_{t\in G_k}g(t)\Bigr)\]
for each $k\in\omega$.

First define $F_0=\{\varepsilon\}$, $G_0=\{\varepsilon\}$ and $f(\varepsilon)=U_1$, $g(\varepsilon)=U_2$.

Assume we have defined $F_k, G_k\subseteq \omega^k$, and $f(s)$ for each $s\in F_k$, and $g(t)$ for each $t\in G_k$. Let
\[\calI = \{ V\in\PO(X) \mid (\exists s\in F_k)[\, V \prec f(s)\,] \text{ or } (\exists t\in G_k)[\, V \prec g(t)\,] \}.\]
The assumption $K\subseteq \left(\bigcup_{s\in F_k}f(s)\right)\cup \left(\bigcup_{t\in G_k}g(t)\right)$ implies that for each $x\in K$ there is $s\in F_k$ with $x\in f(s)$ or there is $t\in G_k$ with $x\in g(t)$. It follows from the definition of a convergent approximation relation that there is $V\in\calI$ with $x\in V$. From the compactness of $K$ there exist $V_0,\ldots,V_n \in \calI$ covering $K$. Define
\begin{eqnarray*}
F_{k+1}&=&\{ s\diamond i\mid s\in F_k \text{ \& } V_i\prec f(s) \},\\
G_{k+1}&=&\{ t\diamond i\mid t\in G_k \text{ \& } V_i \prec g(t) \},
\end{eqnarray*}
and define $f(s\diamond i) = V_i$ for $s\diamond i \in F_{k+1}$ and $g(t\diamond i) = V_i$ for $t\diamond i \in G_{k+1}$.

Let $S = \bigcup_{k\in\omega} F_k$ and $T = \bigcup_{k\in\omega} G_k$. It is clear from the above construction that $S$ and $T$ are finitely branching trees, and $f\colon S \to \PO(X)$ and $g\colon T\to\PO(X)$ are approximation schemes. It follows that the ranges of the functions $\widehat{f}\colon [S]\to X$ and $\widehat{g}\colon [T]\to X$ are compact subsets of $X$. Let $K_1$ be the saturation of the range of $\widehat{f}$ and let $K_2$ be the saturation of the range of $\widehat{g}$. Clearly, $K_1\subseteq f(\varepsilon)=U_1$ and $K_2\subseteq g(\varepsilon)=U_2$.

Fix $x\in K$ and let $W = X\setminus \dnarw \{x\}$. If $K_1\subseteq W$ then for each $p\in [S]$ there is finite $s\sqsubset p$ such that $f(s)\subseteq W$. The set of such $s$ define an open covering of $[S]$, hence the compactness of $[S]$ and the construction of $f$ implies there is $k\in\omega$ such that $f(s)\subseteq W$ for all $s\in F_k$. Similarly, if $K_2\subseteq W$ then there is $k'\in\omega$ such that $g(t)\subseteq W$ for all $t\in G_{k'}$. If we assume, without loss of generality, that $k\geq k'$, then
\[
    K\subseteq \Bigl(\bigcup_{s\in F_k}f(s)\Bigr)\cup \Bigl(\bigcup_{t\in G_k}g(t)\Bigr)\subseteq W,
\]
which contradicts $x \in K\setminus W$. Therefore, it is impossible for both $K_1$ and $K_2$ to be subsets of $W$, hence $x \in K_1\cup K_2$.
\end{proof}

\begin{cor}\label{cor:qpol_KXcoconsonant}
If $X$ is quasi-Polish then $\PU(X)$ is co-consonant.
\end{cor}
\begin{proof}
Fix $\calK\in\PU(\PU(X))$ and $\calU \in \PO(\PU(X))$ such that $\calK\subseteq \calU$. From the compactness of $\calK$ and the definition of the upper Vietoris topology, there exist $U_0,\ldots, U_n \in \PO(X)$ such that $\calK \subseteq \square U_0 \cup \cdots \cup \square U_n \subseteq \calU$. $\PU(X)$ is a Wilker space because it is quasi-Polish, and therefore there exist $\calK_i \in \PU(\PU(X))$ $(0\leq i \leq n)$ such that $\calK_i \subseteq \square U_i$ and $\calK \subseteq \calK_0 \cup \cdots \cup \calK_n$. For each $i$, the set $K_i  = \mu^{\PU}_X(\calK_i) = \bigcup \calK_i$ is in $\PU(X)$ and $K_i \in \square U_i$. Clearly $\calK_i \subseteq \uparw\{K_i\}$, hence $\calK \subseteq \uparw\{K_0, \ldots, K_n\} \subseteq \calU$, which shows that $\calK$ is strongly compact. Since $\PU(X)$ is consonant, the second part of Proposition~\ref{prop:coconsonant_strongcompact} implies $\PU(X)$ is co-consonant.
\end{proof}

\begin{exa}
$\PL(X)$ is not co-consonant in general, even when $X$ is quasi-Polish. As a counterexample, let $X$ be the quasi-Polish space with the weak topology from Example~\ref{ex:AX_not_Scott}. Then $\calK = \Diamond X = \PL(X)\setminus\{\emptyset\}$ is compact (because it is the saturation of the image of the compact space $X$ under the continuous map $\eta^\PL_X$) and it is open in $\PL(X)$, but $\calK$ does not equal the saturation of a finite subset of $\PL(X)$.
\qed%
\end{exa}

\section{\texorpdfstring{$\PL(\PO(X))$}{A(O(X))} and \texorpdfstring{$\PO(\PU(X))$}{O(K(X))} are homeomorphic}\label{sec:AO_OK} 

\begin{defi}\label{def:AO_OK}
For each quasi-Polish space $X$ define $\aook_X\colon \PL(\PO(X))\to\PO(\PU(X))$ and $\okao_X\colon \PO(\PU(X))\to\PL(\PO(X))$ as
\begin{align*}
\aook_X(\frakA)&= \bigcup_{U\in\frakA}\square U,\\
\okao_X(\calU)&= \{U\in\PO(X) \mid \square U \subseteq \calU \}.
\tag*{\qed}
\end{align*}
\end{defi}
We will usually omit the subscripts from $\aook_X$ and $\okao_X$ when the space is clear from the context.

\begin{lem}\label{lem:AO_OKcontinuous}
The function $\aook$ is continuous for each quasi-Polish space $X$.
\end{lem}
\begin{proof}
$\PU(X)$ is co-consonant by Corollary~\ref{cor:qpol_KXcoconsonant}, so it suffices to show that $\aook^{-1}(\triangleup \calA)$ is open for each $\calA\in\PL(\PU(X))$. For $\frakA\in\PL(\PO(X))$ we have
\begin{eqnarray*}
\aook(\frakA) \in \triangleup \calA &\iff& \left(\bigcup_{U\in\frakA}\square U \right)\cap\calA\not=\emptyset\\
&\iff& (\exists U\in\frakA)\,\calA\in\Diamond\Box U\\
&\iff& (\exists U\in\frakA)\, \akka(\calA)\in\Box\Diamond U\text{ (by Lemma~\ref{lem:akka_correct})}\\
&\iff& (\exists U\in\frakA)\,U \in \kaoo(\akka(\calA))\text{ (by definition of $\kaoo$)}\\
&\iff& \frakA \in \Diamond\kaoo(\akka(\calA)),
\end{eqnarray*}
where $\kaoo\circ\akka\colon \PL(\PU(X))\to\PO(\PO(X))$ is the composition of the homeomorphisms from Definitions~\ref{def:AK_KA} and~\ref{def:KAOO}.
\end{proof}

\begin{lem}\label{lem:OK_AOcontinuous}
The function $\okao$ is well-defined and continuous for each quasi-Polish space $X$.
\end{lem}
\begin{proof}
The function is well-defined because the right hand side of the definition is a lower subset of $\PO(X)$ which is closed under directed joins (because $\Box$ distributes over directed joins), and therefore Scott-closed.
For $\calH\in\PO(\PO(X))$ and $\calU\in\PO(\PU(X))$ we have
\begin{eqnarray*}
\okao(\calU) \in \Diamond\calH &\iff&  (\exists U\in\PO(X)) [\, U\in \calH \text{ and } \square U \subseteq \calU\,]\\
&\iff&  (\exists U\in\PO(X)) [\, \ooka(\calH)\in\Box\Diamond U \text{ and } \square U \subseteq \calU\,]\text{ (by Lemma~\ref{lem:OOX->KAX})}\\
&\iff&  (\exists U\in\PO(X)) [\, \kaak(\ooka(\calH))\in\Diamond\Box U \text{ and } \square U \subseteq \calU\,]\text{ (by Theorem~\ref{thrm:consonance})}\\
&\iff& \kaak(\ooka(\calH)) \cap \calU \not=\emptyset\\
&\iff& \calU \in \triangleup \kaak(\ooka(\calH)),
\end{eqnarray*}
where $\kaak\circ\ooka\colon \PO(\PO(X))\to \PL(\PU(X))$ is the composition of the homeomorphisms from Definitions~\ref{def:AK_KA} and~\ref{def:KAOO}. For the fourth equivalence, recall the definition of $\Diamond$ and use the fact that $\calU\subseteq \PU(X)$ is open.
\end{proof}



\begin{thm}\label{thrm:AO_OK_homeomorphic}
If $X$ is quasi-Polish then $\PL(\PO(X))$ with the lower Vietoris topology is homeomorphic (via $\aook$ and its inverse $\okao$) to $\PO(\PU(X))$ with the Scott topology.
\end{thm}
\begin{proof}
From the previous lemmas, it suffices to show that $\aook\circ \okao$ and $\okao\circ \aook$ are the identity functions.

First, for $\calU \in\PO(\PU(X))$ and $\calA\in\PL(\PU(X))$, from the proofs of Lemmas~\ref{lem:AO_OKcontinuous} and~\ref{lem:OK_AOcontinuous} we have
\begin{eqnarray*}
\aook(\okao(\calU)) \in \triangleup\calA &\iff& \okao(\calU) \in \Diamond\kaoo(\akka(\calA))\\
&\iff& \calU \in \triangleup \kaak(\ooka(\kaoo(\akka(\calA))))\\
&\iff& \calU \in \triangleup \calA\text{ (by Theorems~\ref{thrm:KA_OO_iso} and~\ref{thrm:cons_AK_KA_homeo})},
\end{eqnarray*}
hence $\aook\circ \okao$ is the identity on $\PO(\PU(X))$.

Similarly, for $\frakA \in \PL(\PO(X))$ and $\calH\in\PO(\PO(X))$ we have
\begin{eqnarray*}
\okao(\aook(\frakA)) \in\Diamond \calH &\iff& \aook(\frakA)\in\triangleup \kaak(\ooka(\calH))\\
&\iff& \frakA \in \Diamond\kaoo(\akka(\kaak(\ooka(\calH))))\\
&\iff& \frakA \in \Diamond\calH \text{ (by Theorems~\ref{thrm:cons_AK_KA_homeo} and~\ref{thrm:KA_OO_iso})},
\end{eqnarray*}
hence $\okao\circ\aook$ is the identity on $\PL(\PO(X))$.
\end{proof}

The homeomorphisms $\aook$ and $\okao$ determine natural isomorphisms between the contravariant functors $\PL\circ\PO$ and $\PO\circ \PU$ when we restrict the domains of the functors to the subcategory of quasi-Polish spaces. To prove that $\okao_X \circ \PO(\PU(f)) = \PL(\PO(f)) \circ \okao_Y$ for every continuous function $f\colon X\to Y$, we first show that both functions in the equation are frame homomorphisms. It is clear that $\PO(\PU(f))$ is a frame homomorphism, and so are $\okao_X$ and $\okao_Y$ because they are order isomorphisms. Applying the monad $\PL$ to any continuous map results in a join preserving semilattice homomorphism (see Section~6.3 of~\cite{schalk:phd}), hence $\PL(\PO(f))$ preserves all joins. Furthermore, the fact that $\PO(f)$ preserves finite meets lifts\footnote{We will see in the next section that $\PO(X)$ and $\PO(Y)$ are (consonant) algebras of the upper powerspace monad $\PU$. With little effort one can show that the frame homomorphism $\PO(f)$ is a morphism of $\PU$-algebras, and then use general results on Beck distributivity~\cite{be69} to show that $\PL(\PO(f))$ is also a morphism of $\PU$-algebras. It then follows from Schalk's results (see Section~7.3 of~\cite{schalk:phd}) that $\PL(\PO(f))$ preserves finite meets.} to $\PL(\PO(f))$. This is clear for the empty meet, and is also easy to see for binary meets by using the fact that closed sets are lower sets:
\begin{eqnarray*}
\PL(\PO(f))(\frakA_1 \cap \frakA_2) &=& Cl\left(\{ \PO(f)(U) \mid U \in \frakA_1 \cap \frakA_2 \}\right)\\
  &=& Cl\left(\{ \PO(f)(U_1 \cap U_2) \mid U_1 \in \frakA_1 \text{ \& } U_2 \in \frakA_2 \}\right)\\
  &=& Cl\left(\{ \PO(f)(U_1) \cap \PO(f)(U_2) \mid U_1 \in \frakA_1 \text{ \& } U_2 \in \frakA_2 \}\right)\\
  &=& Cl\left(\{ \PO(f)(U_1) \mid U_1 \in \frakA_1 \}\right)\cap Cl\left(\{ \PO(f)(U_2) \mid U_2 \in \frakA_2 \}\right)\\
  &=& \PL(\PO(f))(\frakA_1) \cap \PL(\PO(f))(\frakA_2).
\end{eqnarray*}
To show the reverse inclusion $\supseteq$ of the fourth equality above, note that if $W \in\PO(X)$ is in the intersection of the two closed sets on the right hand side of the equation and $K \subseteq W$ is compact, then there must be $U_1 \in \frakA_1$ and $U_2 \in \frakA_2$ with $K \subseteq \PO(f)(U_1) $ and $K \subseteq \PO(f)(U_2)$, hence $K \subseteq \PO(f)(U_1)\cap\PO(f)(U_2)$. It follows from consonance that every neighborhood of $W$ intersects the closed set on the left hand side of the equation.

So to prove that the frame homomorphisms  $\okao_X \circ \PO(\PU(f))$ and $\PL(\PO(f)) \circ \okao_Y$ are equal, it suffices to show that they both map each basic open $\Box U \in \PO(\PU(Y))$ to $\dnarw\{f^{-1}(U)\}\in \PL(\PO(X))$. Indeed, $\PO(\PU(f)) (\Box U) = {\PU(f)}^{-1}(\Box U) = \Box f^{-1}(U)$, which $\okao_X$ then maps to $\{W \in \PO(X) \mid \Box W \subseteq \Box f^{-1}(U)\} = \dnarw\{f^{-1}(U)\}$. Going in the other direction, $\okao_Y(\Box U) = \{W\in \PO(Y) \mid \Box W \subseteq \Box U\} = \dnarw\{U\}$, and $\PL(\PO(f))(\dnarw\{U\})$ equals the closure of $\{\PO(f)(W) \mid W \subseteq U \}$, which is just $\dnarw\{f^{-1}(U)\}$. The naturality of $\aook$ immediately follows from that of $\beta$ because they are inverses.

\begin{cor}
If $X$ is quasi-Polish then the function $\bigcup\colon \PL(\PO(X)) \to \PO(X)$, which maps a closed subset of opens to their set-theoretical union, is continuous.
\end{cor}
\begin{proof}
Let $\eta^{\PU}_X\colon X\to\PU(X)$ be the unit for the monad $\PU$. The function ${(\eta^\PU_X)}^{-1} \colon \PO(\PU(X)) \to\PO(X)$, which maps each open $\calU\subseteq \PU(X)$ to its preimage under $\eta^\PU_X$, is a frame homomorphism and therefore Scott-continuous. Given $\frakA\in\PL(\PO(X))$ we have
\begin{eqnarray*}
x \in {(\eta^\PU_X)}^{-1}(\aook(\frakA)) &\iff& \eta^\PU_X(x) \in \aook(\frakA)\\
&\iff& \eta^\PU_X(x) \in \bigcup_{U\in\frakA}\square U\\
&\iff& x \in \bigcup_{U\in\frakA} U.
\end{eqnarray*}
Therefore, $\bigcup = {(\eta^\PU_X)}^{-1} \circ \aook$ is continuous.
\end{proof}

It is also interesting to note that the composition $\alpha\circ \eta^\PL_{\PO(X)}$  is the continuous function which maps $U\in\PO(X)$ to $\Box U \in\PO(\PU(X))$.

\begin{cor}\label{cor:OX_sober}
If $X$ is quasi-Polish then $\PO(X)$ is sober.
\end{cor}
\begin{proof}
$\PL(\PO(X))$ is sober by Proposition~\ref{prop:lowerpowerspace_sober}, and the functions $\eta^\PL_{\PO(X)}\colon\PO(X)\to\PL(\PO(X))$ and $\bigcup\colon \PL(\PO(X)) \to \PO(X)$ show that $\PO(X)$ is a continuous retract of $\PL(\PO(X))$.  Therefore, $\PO(X)$ is sober.
\end{proof}

We obtain the following corollary from A.~Schalk's characterization of the algebras of the lower powerspace monad  (see Section~6.3 in~\cite{schalk:phd}).

\begin{cor}
If $X$ is quasi-Polish, then $\PO(X)$ is an algebra of the lower powerspace monad $\PL$, with set-theoretical union as the structure map.
\qed%
\end{cor}

\section{\texorpdfstring{$\PU(\PO(X))$}{K(O(X))} and \texorpdfstring{$\PO(\PL(X))$}{O(A(X))} are homeomorphic}\label{sec:KO_OA} 

\begin{defi}\label{def:KO_OA}
For each quasi-Polish space $X$ define $\kooa_X\colon \PU(\PO(X))\to\PO(\PL(X))$ and $\oako_X\colon \PO(\PL(X)) \to \PU(\PO(X))$ as
\begin{eqnarray*}
\kooa_X(\frakK)&=&\bigcap_{U\in\frakK} \Diamond U,\\
\oako_X(\calU) &=& \{U \in\PO(X) \mid \calU \subseteq \Diamond U\}.
\end{eqnarray*}
\qed%
\end{defi}
We will usually omit the subscripts from $\kooa_X$ and $\oako_X$ when the space is clear from the context.


\begin{lem}\label{lem:KO_OAcontinuous}
The function $\kooa$ is well-defined and continuous for every quasi-Polish space $X$.
\end{lem}
\begin{proof}
We first show that $\kooa(\frakK)$ is an open subset of $\PL(X)$ for each $\frakK\in\PU(\PO(X))$. If $A \in \kooa(\frakK)$, then $A\cap U \not=\emptyset$ for each $U\in\frakK$, hence $\frakK\subseteq \triangleup A$. Propositions~\ref{prop:cntbsd_OX_cons_cocons} and~\ref{prop:coconsonant_strongcompact} together imply $\frakK$ is strongly compact, hence there is finite $\calF\subseteq\PO(X)$ such that $\frakK\subseteq\uparw\calF\subseteq\triangleup A$. Therefore, $A \in \bigcap_{U\in\calF}\Diamond U \subseteq\kooa(\frakK)$, which implies $\kooa(\frakK)$ is open.

Next we show that $\kooa$ is continuous. Since $\PL(X)$ is consonant, it suffices to show that $\kooa^{-1}(\triangledown\calK)$ is open for each $\calK\in\PU(\PL(X))$. For $\frakK\in\PU(\PO(X))$ we have
\begin{eqnarray*}
\kooa(\frakK) \in\triangledown\calK &\iff& \calK \subseteq \bigcap_{U\in\frakK} \Diamond U\\
&\iff& (\forall U\in\frakK)\, \calK\in\Box\Diamond U\\
&\iff& (\forall U\in\frakK)\, U \in \kaoo(\calK)\text{ (by definition of $\kaoo$)}\\
&\iff& \frakK \in \Box\kaoo(\calK),
\end{eqnarray*}
where $\kaoo\colon \PU(\PL(X)) \to \PO(\PO(X))$ is the homeomorphism from Definition~\ref{def:KAOO}.
\end{proof}

\begin{lem}\label{lem:upperset_in_PL}
Let $X$ be a topological space and $\calS \subseteq \PL(X)$ be an upper set. Then for each $A\in\PL(X)$ it holds that $A\in\calS$ if and only if $\calS \not\subseteq \Diamond(X\setminus A)$.
\end{lem}
\begin{proof}
Clearly, if $A\in\calS$ then $\calS \not\subseteq \Diamond(X\setminus A)$. Conversely, if $\calS \not\subseteq \Diamond(X\setminus A)$ then there must be $A'\in\calS$ such that $A' \not\in \Diamond(X\setminus A)$, which implies $A'\subseteq A$. Since $\calS$ is an upper set, we obtain $A\in\calS$.
\end{proof}

\begin{lem}\label{lem:OA_KOcontinuous}
The function $\oako$ is well-defined and continuous for every quasi-Polish space $X$.
\end{lem}
\begin{proof}
In order to prove that $\oako$ is well-defined we must show that $\oako(\calU)\in\PU(\PO(X))$ for each $\calU\in\PO(\PL(X))$. It is easy to see that $\oako(\calU)$ is an upper set, so we only need to show that $\oako(\calU)$ is compact.

Since $X$ is consonant, we only need to consider coverings of $\oako(\calU)$ by basic opens of the form $\triangledown K$ with $K\in\PU(X)$. So assume $\calI\subseteq \PU(X)$ is such that $\oako(\calU) \subseteq \bigcup_{K\in \calI} \triangledown K$. Define $\calA\in\PL(\PU(X))$ to be the closure of $\calI$ in $\PU(X)$. Set $\calK = \akka(\calA)$, where $\akka\colon\PL(\PU(X)) \to \PU(\PL(X))$ is the homeomorphism from Definition~\ref{def:AK_KA}.

We show that $\calK \subseteq \calU$. If $A\in\calK$ then $A\cap K\not=\emptyset$ for each $K\in \calI$, hence $(X\setminus A) \not\in \bigcup_{K\in \calI} \triangledown K$, which implies $(X\setminus A) \not\in \oako(\calU)$. Therefore, $A\in\calU$ by Lemma~\ref{lem:upperset_in_PL} and the definition of $\oako$.

Using the same methods as in the proof of Lemma~\ref{lem:subbasisForAK_KA}, it is easy to see that the basic opens of the form $\bigcap_{U\in\calF}\Diamond U$ for finite $\calF\subseteq \PO(X)$ form a basis for $\PL(X)$ that is closed under finite unions. Therefore, by compactness of $\calK$ there must be $U_0,\ldots,U_n\in\PO(X)$ such that $\calK \subseteq \Diamond U_0\cap \cdots \cap \Diamond U_n \subseteq \calU$. Then for $0\leq i\leq n$, we have $\akka(\calA) = \calK \in \Box\Diamond U_i$, hence $\calA \in \Diamond\Box U_i$ by Lemma~\ref{lem:akka_correct}. From the definition of $\Diamond$ there exist $K_i \in\calA\cap \Box U_i$ for $0\leq i\leq n$, and as $\calA$ is the closure of $\calI$, each $K_i$ can be chosen from $\calI$.

Now we show that $\oako(\calU)\subseteq \triangledown K_0 \cup \cdots \cup \triangledown K_n$. If $U\in\PO(X)$ and $\calU \subseteq \Diamond U$ then $\Diamond U_0 \cap \cdots \cap\Diamond U_n \subseteq \Diamond U$, and hence $\Box\Diamond U_0 \cap \cdots \cap\Box\Diamond U_n \subseteq \Box\Diamond U$ because $\Box$ is monotonic and distributes over finite intersections. Lemma~\ref{lem:akka_correct} yields
\begin{align*}
\Diamond\Box U_0 \cap \cdots \cap\Diamond\Box U_n
    &= \akka^{-1}(\Box\Diamond U_0) \cap \cdots \cap\akka^{-1}(\Box\Diamond U_n) \\
    &= \akka^{-1}(\Box\Diamond U_0 \cap \cdots \cap\Box\Diamond U_n)\subseteq \akka^{-1}(\Box\Diamond U) \\
    &= \Diamond\Box U.
\end{align*}
Since $K_i \in \Box U_i$ we obtain $\dnarw \{K_0,\ldots,K_n\} \in \Diamond \Box U$, hence $K_i \subseteq U$ for some $i\leq n$. Therefore, $U \in \triangledown K_0 \cup \cdots \cup \triangledown K_n$. It follows that $\oako(\calU)$ is compact, hence $\oako$ is well-defined.

Finally, we show that $\oako$ is continuous. Fix $\calH\in\PO(\PO(X))$. Then for $\calU\in\PO(\PL(X))$ we have
\begin{eqnarray*}
\oako(\calU) \in \Box\calH &\iff& (\forall U\in\PO(X))[\,U\in\oako(\calU) \Rightarrow U\in\calH\,]\\
&\iff& (\forall U\in\PO(X))[\,\calU\subseteq \Diamond U \Rightarrow \ooka(\calH)\in\Box\Diamond U\,]\text{ (by Lemma~\ref{lem:OOX->KAX})}\\
&\iff& (\forall U\in\PO(X))[\,\ooka(\calH)\not\subseteq\Diamond U \Rightarrow \calU\not\subseteq \Diamond U \,]\\
&\iff&(\forall A\in\PL(X))[\,\ooka(\calH)\not\subseteq\Diamond (X\setminus A) \Rightarrow \calU\not\subseteq\Diamond (X\setminus A) \,]\\
&\iff&(\forall A\in\PL(X))[\,A\in\ooka(\calH) \Rightarrow A \in\calU\,] \text{ (by Lemma~\ref{lem:upperset_in_PL})}\\
&\iff& \ooka(\calH) \subseteq \calU\\
&\iff& \calU \in \triangledown\ooka(\calH),
\end{eqnarray*}
where $\ooka\colon \PO(\PO(X)) \to \PU(\PL(X))$ is the homeomorphism from Definition~\ref{def:KAOO}.
\end{proof}

\begin{thm}\label{thrm:KO_OA_homeomorphic}
If $X$ is quasi-Polish then $\PU(\PO(X))$ with the upper Vietoris topology is homeomorphic (via $\kooa$ and its inverse $\oako$) with $\PO(\PL(X))$ with the Scott topology.
\end{thm}
\begin{proof}
From the previous lemmas, it suffices to show that $\oako\circ \kooa$ and $\kooa \circ \oako$ are the identity functions.

First, for $\calU \in\PO(\PL(X))$ and $\calK\in\PU(\PL(X))$, from the proofs of Lemmas~\ref{lem:KO_OAcontinuous} and~\ref{lem:OA_KOcontinuous} we have
\begin{eqnarray*}
\kooa(\oako(\calU))\in \triangledown \calK &\iff& \oako(\calU) \in \Box\kaoo(\calK)\\
&\iff& \calU \in \triangledown\ooka(\kaoo(\calK))\\
&\iff& \calU \in \triangledown\calK\text{ (by Theorem~\ref{thrm:KA_OO_iso})},
\end{eqnarray*}
hence $\kooa\circ \oako$ is the identity on $\PO(\PL(X))$.

Similarly, for $\frakK \in \PU(\PO(X))$ and $\calH\in\PO(\PO(X))$ we have
\begin{eqnarray*}
\oako(\kooa(\frakK)) \in\Box\calH &\iff& \kooa(\frakK) \in \triangledown\ooka(\calH)\\
&\iff& \frakK \in \Box\kaoo(\ooka(\calH))\\
&\iff& \frakK \in \Box\calH\text{ (by Theorem~\ref{thrm:KA_OO_iso})},
\end{eqnarray*}
hence $\oako\circ\kooa$ is the identity on $\PU(\PO(X))$.
\end{proof}

The homeomorphisms $\kooa$ and $\oako$ determine natural isomorphisms between the contravariant functors $\PU\circ\PO$ and $\PO\circ \PL$ when restricted to quasi-Polish spaces. The proof is essentially the same as the proof for $\aook$ and $\okao$ in Section~\ref{sec:AO_OK}. To prove that $\oako_X \circ \PO(\PL(f)) = \PU(\PO(f)) \circ \oako_Y$ for every continuous $f\colon X\to Y$, we first show that both functions in the equation are frame homomorphisms. The non-trivial part of this claim is showing that $\PU(\PO(f))$ preserves binary joins.\footnote{We will give a direct topological proof here, but alternatively one could use Beck distributivity to show that $\PU(\PO(f))$ is a morphism of $\PL$-algebras, and then use Schalk's results (see Section~6.3 of~\cite{schalk:phd}) to show that $\PU(\PO(f))$ preserves joins.} We get an explicit description of binary joins in $\PU(\PO(X))$ by using the order isomorphism $\kooa$ to pass to $\PO(\PL(X))$, where joins are unions, and then pull the result back to $\PU(\PO(X))$ via the order isomorphism $\oako$:
\begin{eqnarray*}
\frakK_1 \vee \frakK_2 &=& \oako(\kooa(\frakK_1) \cup \kooa(\frakK_2))\\
 &=& \{U \in\PO(X) \mid \kooa(\frakK_1) \cup \kooa(\frakK_2) \subseteq \Diamond U\}\\
 &=& \{U \in\PO(X) \mid \kooa(\frakK_1) \subseteq \Diamond U \text{ \& } \kooa(\frakK_2) \subseteq \Diamond U\}\\
 &=& \{U \in\PO(X) \mid \kooa(\frakK_1) \subseteq \Diamond U\} \cap \{U \in\PO(X) \mid \kooa(\frakK_2) \subseteq \Diamond U\}\\
 &=& \oako(\kooa(\frakK_1)) \cap \oako(\kooa(\frakK_2))\\
 &=& \frakK_1 \cap \frakK_2.
\end{eqnarray*}
Thus binary joins in $\PU(\PO(X))$ correspond to intersections, and similarly for $\PU(\PO(Y))$. The rest of the proof that $\PU(\PO(f))$ preserves binary joins is dual to the proof that $\PL(\PO(f))$ preserves binary meets, and this time uses the fact that saturated compact sets are upper sets and that $\PO(f)$ preserves binary joins:
\begin{eqnarray*}
\PU(\PO(f))(\frakK_1 \cap \frakK_2) &=& \uparw\{ \PO(f)(U) \mid U \in \frakK_1 \cap \frakK_2 \}\\
  &=& \uparw\{ \PO(f)(U_1 \cup U_2) \mid U_1 \in \frakK_1 \text{ \& } U_2 \in \frakK_2 \}\\
  &=& \uparw\{ \PO(f)(U_1) \cup \PO(f)(U_2) \mid U_1 \in \frakK_1 \text{ \& } U_2 \in \frakK_2 \}\\
  &=& \uparw\{ \PO(f)(U_1) \mid U_1 \in \frakK_1 \} \cap \uparw\{ \PO(f)(U_2) \mid U_2 \in \frakK_2 \}\\
  &=& \PU(\PO(f))(\frakK_1) \cap \PU(\PO(f))(\frakK_2).
\end{eqnarray*}
So to prove that the frame homomorphisms $\oako_X \circ \PO(\PL(f))$ and $\PU(\PO(f)) \circ \oako_Y$ are equal, it suffices to show that they both map each subbasic open $\Diamond U \in \PO(\PL(Y))$ to $\uparw\{f^{-1}(U)\}\in \PU(\PO(X))$. Indeed, $\PO(\PL(f))(\Diamond U) = {\PL(f)}^{-1}(\Diamond U) = \Diamond f^{-1}(U)$, which $\oako_X$ then maps to ${\{W\in\PO(X)\mid \Diamond f^{-1}(U) \subseteq \Diamond W\} = \uparw\{f^{-1}(U)\}}$. Going in the other direction, $\oako_Y(\Diamond U) = \{ W \in \PO(Y) \mid \Diamond U \subseteq \Diamond W\} = \uparw\{U\}$, and $\PU(\PO(f))(\uparw\{U\})$ equals the saturation of $\{\PO(f)(W) \mid U \subseteq W\}$, which is just $\uparw\{f^{-1}(U)\}$.

\begin{cor}\label{cor:compact_intersections}
If $X$ is quasi-Polish then the function $\bigcap\colon \PU(\PO(X)) \to \PO(X)$, which maps a compact subset of opens to their set-theoretical intersection, is well-defined and continuous.
\end{cor}
\begin{proof}
Let $\eta^\PL_X\colon X\to\PL(X)$ be the unit for the monad $\PL$. The function ${(\eta^\PL_X)}^{-1} \colon \PO(\PL(X)) \to\PO(X)$, which maps each open $\calU\subseteq \PL(X)$ to its preimage under $\eta^\PL_X$, is a frame homomorphism and therefore Scott-continuous. Given $\frakK\in\PU(\PO(X))$ we have
\begin{eqnarray*}
x \in {(\eta^\PL_X)}^{-1}(\kooa(\frakK)) &\iff& \eta^\PL_X(x) \in \kooa(\frakK)\\
&\iff& \eta^\PL_X(x) \in \bigcap_{U\in\frakK} \Diamond U\\
&\iff& x \in \bigcap_{U\in\frakK} U.
\end{eqnarray*}
Therefore, $\bigcap = {(\eta^\PL_X)}^{-1} \circ \kooa$ is well-defined and continuous.
\end{proof}

It is also interesting to note that the composition $\gamma\circ \eta^\PU_{\PO(X)}$ is the continuous function which maps $U\in\PO(X)$ to $\Diamond U \in\PO(\PL(X))$.

We obtain the following corollary from A.~Schalk's characterization of the algebras of the upper powerspace monad  (see Section~7.3 in~\cite{schalk:phd}).

\begin{cor}
If $X$ is quasi-Polish, then $\PO(X)$ is an algebra of the upper powerspace monad $\PU$, with set-theoretical intersection as the structure map.
\qed%
\end{cor}

\section*{Acknowledgments}
This paper benefited greatly from multiple discussions with Klaus Keimel, who unfortunately passed away while this paper was being reviewed. Klaus was a great mathematician, a kind friend, and an important source of encouragement. He will be deeply missed.

We thank Reinhold Heckmann for many helpful comments on an earlier version of this paper, and thank Matthias Schr\"{o}der for useful discussions, in particular for sharing his insights about the Scott-topology on the open set lattice of a sequential space. We also thank the reviewers for carefully reading the paper and providing useful suggestions for improvement. This work was supported by JSPS Core-to-Core Program, A. Advanced Research Networks, and the first author was supported by JSPS KAKENHI Grant Numbers 15K15940 and 18K11166.

\bibliographystyle{alpha}
\bibliography{Powerspaces}

\end{document}